\documentclass[11pt]{amsart}

 \makeatletter\def\journal@name{}\makeatother

\RequirePackage[OT1]{fontenc}
\RequirePackage{amsthm,amsmath}
\RequirePackage[colorlinks,citecolor=blue,urlcolor=blue]{hyperref}
 \usepackage{enumerate}

\usepackage[square,compress,comma, numbers,sort]{natbib}
\usepackage{amsfonts,mathtools}

\numberwithin{equation}{section}
\theoremstyle{plain}

\allowdisplaybreaks[4]

\def\HA{\mathcal{H}}

\usepackage{amssymb}
\usepackage{color}
\newcommand{\abs}[1]{\left\lvert #1 \right\rvert}
\newcommand{\Abs}[1]{ \biggl \lvert #1 \biggr \rvert}
\newcommand{\ABs}[1]{ \biggl \lvert #1 \biggr \rvert}

\DeclarePairedDelimiterXPP\pk[1]{\mathbb{P}}\{ \}{}{ #1}
\DeclarePairedDelimiterXPP\E[1]{\mathbb{E}}\{ \}{}{	#1}

\usepackage{xparse}

\NewDocumentCommand{\ceil}{s O{} m}{%
	\IfBooleanTF{#1} 
	{\left\lceil#3\right\rceil} 
	{#2\lceil#3#2\rceil} 
}
\NewDocumentCommand{\floor}{s O{} m}{%
	\IfBooleanTF{#1} 
	{\left\lfloor#3\right\rfloor}
	{#2\lfloor#3#2\rfloor}
}

\def\HH{{[h]}}

\def\x{\vk{x}}

\definecolor{c20}{rgb}{0.,0.7,0.}
\definecolor{c30}{rgb}{0.,0.,1.}
\definecolor{c40}{rgb}{1,0.1,0.7}
\definecolor{c50}{rgb}{1,0,0}
\definecolor{c60}{rgb}{1,0.9,0.1}
\definecolor{c70}{rgb}{0.50,1.00,0.00}

\def\kk#1{{\textcolor{black}{#1}}}
\def\k1#1{{\textcolor{red}{#1}}}

\numberwithin{equation}{section}
\newtheorem{theo}{Theorem}[section]
\newtheorem{sat}[theo]{Proposition}
\newtheorem{de}[theo]{Definition}
\newtheorem{lem}[theo]{Lemma}

\newtheorem{example}[theo]{Example}
\newtheorem{korr}[theo]{Corollary}
\newtheorem{remark}[theo]{Remark}

\newtheorem{prop}[theo]{Proposition}

\numberwithin{equation}{section}

\newcommand{\prooflem}[1]{\textsc{Proof of Lemma} \ref{#1}:}
\newcommand{\proofkorr}[1]{\textsc{Proof of Corollary} \ref{#1}:}

\newcommand{\QED}{\hfill $\Box$}

\newcommand{\COM}[1]{}

\def\IF{\infty}

\newcommand{\R}{\mathbb{R}}
\newcommand{\inr}{\in \R}

\newcommand{\BQN}{\begin{eqnarray}}
\newcommand{\EQN}{\end{eqnarray}}
\newcommand{\BQNY}{\begin{eqnarray*}}
	\newcommand{\EQNY}{\end{eqnarray*}}
\def\ldot{, \ldots,}

\newcommand{\limit}[1]{\lim_{#1 \to   \infty}}

\newcommand{\kb}[1]{\boldsymbol{#1}}
\newcommand{\vk}[1]{\kb{#1}}

\def\bqn#1{\begin{eqnarray} #1 \end{eqnarray}}

\newcommand{\BS}{\begin{sat}}
	\newcommand{\ES}{\end{sat}}
\newcommand{\BT}{\begin{theo}}
	\newcommand{\ET}{\end{theo}}
\newcommand{\BK}{\begin{korr}}
	\newcommand{\EK}{\end{korr}}

\newcommand{\BEX}{\begin{example}}
	\newcommand{\EEX}{\end{example}}

\newcommand{\BD}{\begin{de}}
	\newcommand{\ED}{\end{de}}

\newcommand{\BIT}{\begin{itemize}}
	\newcommand{\EIT}{\end{itemize}}
\newcommand{\BDI}{\begin{description}}
	\newcommand{\EDI}{\end{description}}

\newcommand{\BRM}{\begin{remark}}
	\newcommand{\ERM}{\end{remark}}

\newcommand{\BEL}{\begin{lem}}
	\newcommand{\EEL}{\end{lem}}

\newcommand{\nelem}[1]{{Lemma \ref{#1}}}
\newcommand{\neprop}[1]{{Proposition \ref{#1}}}
\newcommand{\netheo}[1]{{Theorem \ref{#1}}}

\newcommand{\nekorr}[1]{{Corollary \ref{#1}}}

\def\X{\vk{X}}

\def\X{\vk{X}}

\def\Z{\mathbb{Z}}
\def\inn{\in \mathbb{N}}

\def\bD{ \mathbf{ D} }

\def\Var{\mathrm{Var}}
\def\H{{H}}

\def\LT{\left}
\def\RT{\right}

\def\bF{ \bar F}

\definecolor{c20}{rgb}{0.,0.7,0.}
\definecolor{c30}{rgb}{0.,0.,1.}
\definecolor{c40}{rgb}{1,0.1,0.7}
\definecolor{c50}{rgb}{1,0,0}
\definecolor{c60}{rgb}{1,0.9,0.1}

\def\VB{ \sigma^2_{\vk b} }

\definecolor{kd}{RGB}{255,0,0}
\definecolor{eh}{RGB}{0,255,0}
\definecolor{lw}{RGB}{0,0,255}

\def\b{\vk{b}}
\def\x{\vk{x}}
\def\tilb{\widetilde{\b}}
\def\IB{I}
\def\JB{J}
\def\SI{\Sigma}

\def\SIJI{\Sigma_{JI}}
\def\SIM{\SI^{-1}}

\def\SIIIM{ (\Sigma_{II})^{-1} }

\def\x{\vk{x}}
\def\y{\vk{y}}

\def\bD{ E }

\def\t{t}
\def\xxiu{\vk{X}_{u,\tau}}

\def\w{\vk{w}}

\def\rmF{\mathrm{F}}
\newcommand{\normF}[1]{\left\Vert#1\right\Vert_{\mathrm{F}}}

\usepackage{mathtools}

\newcommand{\ve}{\varepsilon}

\newcommand{\sP}{\mathcal{P}}

\DeclareMathOperator*{\sgn}{sgn}
\DeclareMathOperator*{\cov}{cov}

\newcommand{\diag}[1]{\mathrm{diag}(#1)}
\def\Y{\vk{Y}}
\def\v{\vk{v}}
\def\X{\vk{X}}
\def\id{\mathbf{1}}
\def\Z{\vk{Z}}
\def\W{\vk{W}}

\def\d{\vk{d}}
\def\I{\mathcal{I}}

\def\IF{\infty}

\def\LT{\left}
\def\RT{\right}

\def\HH{\mathcal{H}}

\def\pB{p_{\vk b} (u)}

\begin{document}

\title{Extremes of Vector-Valued Gaussian Processes}

\author{Krzysztof D\c{e}bicki}
\address{Krzysztof D\c{e}bicki, Mathematical Institute, University of Wroc\l aw, pl. Grunwaldzki 2/4, 50-384 Wroc\l aw, Poland}
\email{Krzysztof.Debicki@math.uni.wroc.pl}

\author{Enkelejd  Hashorva}
\address{Enkelejd Hashorva, Department of Actuarial Science, 
University of Lausanne,
UNIL-Dorigny, 1015 Lausanne, Switzerland
}
\email{Enkelejd.Hashorva@unil.ch}

\author{Longmin Wang }
\address{Longmin Wang, School of Mathematical Sciences, Nankai University, 94 Weijin Road, Tianjin 300071, P.R. China
}
\email{wanglm@nankai.edu.cn}
 \maketitle
\begin{abstract}
The seminal  papers of Pickands \cite{PicandsB,MR0250368} paved the way for a systematic study of high exceedance probabilities
of both stationary and non-stationary Gaussian processes.
Yet, in the vector-valued setting, due to the lack of key tools including  Slepian's Lemma,
Borell-TIS and Piterbarg inequalities there has not been any methodological development in the
literature for the study of extremes of vector-valued Gaussian processes. In this contribution we develop the uniform double-sum method for the vector-valued setting
obtaining the exact asymptotics of the exceedance probabilities for both stationary and non-stationary Gaussian processes.
We apply our findings to 
the  operator fractional Brownian motion and the operator fractional Ornstein-Uhlenbeck process.
\end{abstract}
\bigskip

\bigskip


\section{Introduction}
The asymptotic analysis  of probabilities of rare events has been the topic of numerous past contributions and is still
an active area of research. In this article the rare events of interests are the high exceedances of
vector-valued Gaussian processes, i.e., we shall investigate  the {exact} approximation of
$$\pB = \pk{ \exists t \in [0,T]:  X_j(t) > ub_j,j\le d }$$
as $u \to \IF$ with $\vk X(t)= (X_1(t) \ldot X_d(t))^\top$, $t\in [0,T]$ a given centered $\R^d$-valued Gaussian process
with a.s.\ continuous sample paths and given constants $b_i$'s.  In order to avoid trivialities, hereafter we shall assume that at least one of the  $b_i$'s  is positive. \\

The approximation of  $\pB$ is of interest in various applications including statistics,   ruin theory, queueing theory,
see e.g., \cite{AzV19, MR1747100, 
	borovkov17, PalmBorov, delsing2018asymptotics}. 
 Large deviation type results related to {the vector-valued setting of this contribution}
 are obtained in \cite{PS05,Debicki10},  see \cite{MR1272734,MR1415241, MR3336847} for various interesting findings for non-Gaussian $\vk X$.

 Even the seemingly trivial case that $\vk X$ has independent components is quite challenging,
 see the recent contributions by Azais and Pham \cite{AzV19, Vie19}.  The   available  results in the literature that concern
 Gaussian processes with  dependent components cover only linear transformations of an $\R^d$-valued
 Brownian motion,  see \cite{Rolski17,mi:18}.
 The independence of increments and the self-similarity property of Brownian motion are essential
 properties 
 used in the aforementioned contributions,
 {which also determine limits of applicability of methods used in \cite{Rolski17,mi:18}}.\\

In the one-dimensional case three different methods are often utilised when dealing with
the asymptotics of extremes of Gaussian processes: Pickands method which is based on the
discretisation of supremum  and the negligibility of double-sum term \kk{(see also \cite{ChanLai} for further refinements)},
Piterbarg's approach  which makes particular use of continuous mapping theorem \kk{(see \cite{Pit20})},
and Berman's method  that capitalises on the relation between supremum and sojourn times \kk{(see \cite{Berman82, Berman92})}.
{All the above techniques are}
heavily based on the following fundamental results
\begin{enumerate}
\item [i)] Slepian lemma, see e.g., \cite[Thm 2.2.1]{AdlerTaylor}; 
\item[ii)] Borell-TIS and Piterbarg inequalities, see e.g., \cite[Thm 2.1.1]{AdlerTaylor} and \cite[Thm 8.1]{Pit2001};
\item [iii)] {uniform version} of the classical Pickands-Piterbarg lemma, see \cite[Lem 2.1]{DeHaJi2016}. 
\end{enumerate}
Roughly speaking, in the one-dimensional setting,  {Slepian  lemma},
Borell-TIS and Piterbarg inequalities are essential for the non-stationary case.
The first {one} is utilised to approximate rare events by switching to  stationary Gaussian processes,
whereas {the} both inequalities show that only a small neighbourhood around the point of the maximum
of the variance {(assumed to be unique)} is responsible for the rare-event approximation; {see, e.g., the seminal monograph
by Piterbarg \cite{Pit96}}. \\

{One of the reasons for the lack of} methodological approach for studying  extremes
of vector-valued Gaussian process is that the three \kk{key tools}
mentioned above are not available in the general vector-valued setting. 
In fact, {the existing extensions of Slepian lemma in the form of  Gordon inequality,
are not generally applicable} in higher
dimensions (apart from very special cases like processes with independent components, see e.g., \cite[Lem 5.1]{DHJT15}),
whereas an extension of Borell-TIS and Piterbarg inequalities requires a deep understanding of the problem at hand,
which has been {addressed}  in this paper.

In this contribution  exact asymptotics of $p_{\vk b}(u)$  as $u\to \IF$ for both stationary and non-stationary
$\vk X$ {are derived} by {levering} the uniform double-sum method to the vector-valued setting.
The key to the methodology {{developed} in this contribution} is what we refer to as
{the uniform} Pickands-Piterbarg lemma, see \nelem{Lem1} below.
\\
{We briefly explain the {main} ideas underlying the approach taken in this paper} pointing  out some
subtle issues related to uniform approximations that appear to have been overlooked  in the literature; \cite{Uniform2016} takes  particular care of {those} issues in the one-dimensional setting.

The main attempts of the  double-sum method consist in proving  that
\bqn{ \label{SSA}
	p_{\vk b}(u) \sim \sum_{i=1}^{N_u} \pk{\exists t \in T_i(u) :\ X_j(t)> u b_j,\, j\le d}=: \Sigma(u)
}	
as $u\to \IF$, where $T_k(u)$, $k\le N_u$ are disjoint compact intervals covering $[0,T]$.\\

Commonly,  $\Sigma(u)$ is referred to as the {\it single-sum {term}}.  Each term of the single-sum, say the $j$th one,  is approximated  by
{some} function $\theta_j(u)$ as $u\to \IF$. However, {for non-stationary processes},
such approximation {does not imply}
$\Sigma(u) \sim  \sum_{j=1}^{N_u} \theta_j(u)$  as $u\to \IF$, since typically $N_u$ tends to infinity as $u\to \IF$. This holds true if the aforementioned approximation of $\theta_j(u)$ is {uniform for all positive integers} $j\le N_u$.

In the literature this fact has been not taken care of systematically;
 a notable exception is \cite{MR2111193}. As a result numerous proofs in the literature have certain gaps.
 We take special care of this key uniformity issue by deriving  a uniform version of the Pickands-Piterbarg lemma,
 see  \nelem{Lem1}. {Recently, for the one-dimensional setting,  an alternative approach that solves previous gaps in the literature concerning  uniformity issues has been suggested in \cite{EHP}. However, due to lack of Slepian lemma for general vector-valued Gaussian processes that approach is not  applicable for the studies  of this contribution.} \\

In view of  Bonferroni inequality, $\Sigma(u)$ is an upper bound for   $p_{\vk b}(u)$ and a lower bound
{
is given by}
$ \Sigma(u)- \Sigma \Sigma(u)$ where the so-called {\it double-sum term} is  given by
$$ \Sigma \Sigma(u)= \sum_{i=1}^{N_u}{\sum_{i< j\le N_u}} \pk{\exists (s,t) \in T_i(u)\times T_j(u) :\ X_k(s)> ub_k,\, X_l(t))> u b_l,\, k, \,l\le d}.$$
Showing the asymptotic negligibility of \kk{$\Sigma \Sigma(u)$} as $u\to \IF$ is typically \kk{a hard and technical problem},
since  the asymptotic bounds derived for its summands need also be 
uniform for all positive integers $i$, $j\le N_u$.

{A subtle novelty of our approach for non-stationary $\vk X$ is that we do not use the common approach to standardise
the process and then substitute it by a stationary process (utilising Slepian  lemma).
The reason is that, as previously mentioned, Slepian  lemma does not hold in general for vector-valued Gaussian processes.}\\

An application of our findings concerns the study of the tail behaviour of supremum  of operator fractional Brownian motion (fBm)
discussed briefly  below (see for details Section \ref{s.fbm}).
{Let $H$ be a $d\times d$ real-valued matrix with eigenvalues $h_i\in (0, {1]}$, $i\le d$.}
A centered, sample continuous $\R^d$-valued Gaussian process $\X(t)$, $t\in \R $ is said to be an \emph{operator fBm} with index $H$, if it has stationary increments and is operator self-similar in the sense that
\bqn{\label{OFB}  \{ \X(\lambda t),\ t \in \R \} \stackrel{d}{=} \left\{ \sum_{k=0}^{\infty} \left( \log \lambda  \right)^k \frac{H^{k}}{k!} \X(t),\ t \in \R \right\}
}
for any $\lambda>0$, where $\stackrel{d}{=}$ stands for the equality of finite dimensional distributions (fidi's).
Let $h_*= \min_{1\le i \le d} h_i$.
If further  $\X$ is time-reversible, i.e., $\E{ \X(t) \X(s)^{\top}} = \E { \X(s) \X(t)^{\top} }$
for all $t$ and $s$, 
then in view of \neprop{OPFB} {in Section \ref{s.examples},} as $u\to \IF$, for positive $b_i$'s
$$ p_{\vk b}(u) \sim  C u^{\max(0, \frac{1 - 2 h_{*}}{h_{*}})} \pk{ X_j(T)> u b_j, j\le d}  .  $$
Throughout  this paper $\sim$ means {the} asymptotic equivalence as $u\to \IF$.
If $h_*< 1/2$, then $C$  is given in the form of Pickands-type constant and for $h_*=1/2$ it corresponds to the so-called
Piterbarg-type constant.

{
Other applications {illustrating findings of this contribution} are concerned with  stationary $\vk X$  being the Lamperti transform of
  some operator fBm or  $\vk X$ being  an operator fractional Ornstein-Uhlenbeck (fO-U) process.} \\


 \textbf{Brief organisation of the paper.}  {Main results of this paper} are presented in Section~\ref{s:extreme} with
proofs relegated to  Section~\ref{Secproofs}. We dedicate Section~\ref{s.examples} some important examples and then present in
Section~\ref{s.aux}
several auxiliary results; their proofs are relegated to Appendix.
We conclude this section by introducing  some standard notation. \\

 \textbf{Notation. } All vectors in $\R^d$ are written in bold letters, for instance
 $\b = (b_1, \ldots, b_d)^{\top}$, $\vk{0} = (0, \ldots, 0)^{\top}$ and $
 \id = (1, \ldots, 1)^{\kk{\top}}$. For two vectors $\x$ and $\y$, we write $\x > \y$ if $x_i > y_i$ for all $1 \leq i \leq d$. Given  a  real-valued matrix $A$ we shall write $A_{IJ}$ for the submatrix of $A$ determined by keeping the rows and columns of $A$ with row indices in the non-emtpy set $I$ and column indices in the non-empty set $J$, respectively. If $A$ is a $d\times d$ matrix, then $\normF{A} = \sqrt{\sum_{1 \leq i, j \leq d} a_{ij}^2}$ denotes its  Frobenius norm. {In our notation $\mathcal{I}_d$ is the $d\times d$ identity matrix and $\diag{\x}=\diag{x_1,\ldots,x_d}$ stands for the diagonal matrix
with entries $x_i$, $i=1,\ldots,d$ on the main diagonal, respectively. }

Let in the sequel  $\Sigma \inr^{d \times d}$ be a positive definite  matrix with inverse $\SIM$. If  $\b \inr^d \setminus (-\IF, 0]^d $, then the quadratic programming problem $\Pi_\Sigma(\b)$
\begin{equation}
\label{e:QP}
\Pi_\Sigma(\b)= \text{minimise $ \x^\top \Sigma^{-1} \x $ under the linear constraint } \x \ge \b
\end{equation}
has a unique solution $\tilb \ge \vk b$ and there exists a unique non-empty index set $I\subset \{1 \ldot d\}$ such that
\begin{equation}
\label{keyWW}
\tilb_I=\b_I,  \quad \tilb_J = \SI_{IJ}\SIIIM \b_I\ge  \vk b_J,  \quad \vk w_I= (\Sigma_{II})^{-1}  \b_I> \vk 0_I, \quad \vk w_J= \vk{0}_J
\end{equation}
and  $\vk w= \Sigma ^{-1} \tilb$,
where   {coordinates $J= \{ 1 \ldot d\} \setminus I$ (which can be empty)
are responsible for dimension-reduction phenomena, while coordinates belonging to $I$ play
essential role in the exact asymptotics}.
{We refer to  \nelem{AL} below} for more details. \\ 

Throughout this paper we use the lower case constants $c_1$, $c_2$, $\ldots$ to denote generic constants used in the proofs, whose exact values are not important and can be changed line to line. The labeling of the constants starts anew in every proof.

\section{Main Results}
\label{s:extreme}

As mentioned in the Introduction, we are interested in the exact asymptotics of
\begin{equation}
  \label{e:pbu}
  p_{\vk b}(u)=
  \pk{ \exists t \in [0, T]:\ \X(t) > u \b }, \quad u \to \infty
\end{equation}
for a centered $\R^d$-valued Gaussian process $\X(t)$, $t \in [0, T]$ and any $\b \in \R^d$ with at least
one positive component.
We \kk{shall} state 
our main results for $\X$ stationary and non-stationary separately.

Hereafter for  Gaussian processes defined on some compact parameter set $E\subset \R^k$ we shall assume 
its  a.s. sample continuity.
 Let $\vk Y(t)$, $t\in \bD$ be a   centered $d$-dimensional  vector-valued Gaussian process  i.e.,
$\vk Y(t) =( Y_1(t) \ldot Y_d(t))^\top$, $\t\in \bD$ is a column vector Gaussian process.
Let for any $t,\, s \in \bD$
$${R}(t,s) = \E{\vk Y(t) \vk Y(s) ^\top} $$
be the  covariance matrix function (cmf)
of $\vk Y $, which
is a matrix-valued non-negative definite function in the sense that
 $$\sum_{i, j = 1}^n \v_i^{\top} R(t_i, t_j) \v_j \geq 0$$
  for any $t_i \in \bD, \v_i \in \R^d, i\le n$. Conversely, if  $R$: $\bD \times \bD \mapsto \R^{d\times d}$ is a matrix-valued non-negative definite function such that $R(t,s)
= R(s,t)^\top$ for any $t$, $s \in \bD$,
then there exists an
$\R^d$-valued Gaussian process  $\Y(t)$, $t\in \bD$ with cmf  $R$.
Note that in the definition of positive definite or non-negative definite matrices we do not require the matrices to be symmetric. \\

Let $V$ be a $d \times d$ real-valued matrix and let $\alpha \in (0,2]$ be given. An interesting  example of a cmf  determined by
$V$ is 
\def\SA#1#2{S_\alpha ( #1, #2) }
\begin{equation}
\label{e:KV}
R_{\alpha, V}(t,s) = \SA{t}{V} + \SA{-s}{V}  - \SA{t-s}{V}, \quad t,\, s \in \R,
\end{equation}
with
\[
\SA{t}{V} = |t|^{\alpha} \left( V \id_{\{t \geq 0\}} + V^{\top} \id_{\{t < 0\}} \right) = |t|^{\alpha} \left( V^+ + V^{-} \sgn(t) \right)
\]
and
$$V^+ = \frac{1}{2} \left( V + V^{\top} \right), \quad
V^{-} = \frac{1}{2} \left( V - V^{\top} \right).$$
Note that we use the standard notation  $\sgn(t) = 1$ for $t \geq 0$ and $\sgn(t)=-1$ for $t < 0$.\\
{By \cite[Prop~9]{ACLP13}}, the matrix-valued function $R_{{\alpha, V}}$ defined by \eqref{e:KV} is a non-negative definite function if and only if the Hermitian matrix
\begin{equation}
\label{e:posV}
{V^\star=}\sin \left( \frac{\pi \alpha}{2} \right)\, V^+ - \sqrt{-1} \cos \left( \frac{\pi \alpha}{2} \right)\, V^{-} 
\end{equation}
is non-negative definite. Furthermore, under {the above} conditions, {one can define}
a multivariate fBm  $\Y(t), t\inr$ with $R_{\alpha, V}$ as its cmf. The classical Pickands constants (see \cite{PicandsB})
are defined in terms of a standard fBM. {A multidimensional analog of Pickands constants can be} defined utilising
$\vk Y$. Specifically, for compact $E \subset \R$ set 
\def\PICSW{{{\mathcal{H}}_{\alpha, V}}}
\def\PICS{{{\mathcal{H}}_{\alpha, V_{\w}}}}
\begin{equation}
\label{e:pickandsE}
{{\mathcal{H}}_{\alpha, V}}(E) = \int_{\R^d} e^{\id^{\top} \x} \pk{\exists t \in E:\ \Y(t) -
	{ \SA{t}{V} \id } > \x} d \x.
\end{equation}

As shown in the proof of Theorem~\ref{T:stat}, the function $t \mapsto {{\mathcal{H}}_{\alpha, V}}([0, t])$, $t>0$ is sub-additive, which implies that 
\begin{equation}
\label{e:Pickands}
 \HH_{\alpha, V} = \lim_{T \to \infty} \frac{{\mathcal{H}}_{\alpha, V}([0, T])}{T} = \inf_{T>0}  \frac{{\mathcal{H}}_{\alpha, V}([0, T])}{T} {\in {[}0,\, \infty)}.
\end{equation}

We shall call $\HH_{\alpha, V} $ the multidimensional Pickands constant.

\def\PicK{ \H_{\Y- \vk{d}, \Sigma, \vk b} }
\def\PicKA{ \H_{ \Sigma, \vk b} }

\def\uY { {\overline{\vk{u}}}}

\def\VV{\bD}

\subsection{Stationary case} \label{s:stat}

\def\RR{\mathcal{R}}
Let   $\vk X(t)$, $t\in [0,T] $ be a centered, $\R^d$-valued stationary Gaussian process with cmf  $R(t,s)$.
The stationarity of $\vk X$ means that  $R(s+t,s)=R(t,0)=: {\RR}(t)$ for any $s$, $t$, $s+t \in [0,T]$.
Letting  $\Sigma = \RR(0)$ we have that {for each fixed $t \in {[-T,T]}$},
the matrix $\Sigma - \RR(t)$ is  non-negative definite but not  {necessarily symmetric}, {which is
{reflected in the formula}
\BQNY
\  \E{ \left[ \X(t)
	-   \X(0)
	\right] \left[ \X(t)
	-  \X(0)
	\right]^{\top} }
&=& \Sigma - \RR(t)+ \Sigma- \RR(t)^\top, \quad t\in [0,T].
\EQNY
}

Hereafter $I$ stands for  the unique non-empty subset of $\{1\ldot d \}$  that determines the solution $\tilb$ of
the quadratic programming problem $\Pi_\Sigma(\vk b)$ (defined in the Introduction)
and $\vk w = \Sigma^{-1} \tilb$ has {non-negative components.}\\

In this section we shall impose the following assumptions:
\begin{enumerate}[({B}1)]
	\item\label{I:B1} $\Sigma_{II} - \RR_{II}(t)$ is positive definite for every $t {
	\in (0,T]}$;
	\item \label{I:B2} There exists a $d\times d$ real matrix $V$ such that
	${ \vk{w}^\top V \vk {w}> 0}$  and further  
	\begin{equation}
	\label{e:stat}
	\Sigma - \RR(t) \sim    t^{\alpha} V
	\quad \text{as } t \downarrow 0
	\end{equation}
{	holds for some   $\alpha \in (0,2]$}.
\end{enumerate}
Note in passing that since   $\RR(-t) = \RR(t)^{\top}$, by   \eqref{e:stat} we have that
\[
\Sigma - \RR(t) \sim |t|^{\alpha} V^{\top}, \quad t \uparrow 0.
\]
{Moreover, sufficient and necessary condition for $V$
to satisfy \eqref{e:stat} by a stationary Gaussian process
$\X(t)$, $t\in [0,T]$,
is {that} $V^\star$ given by \eqref{e:posV} is non-negative definite.
}


\begin{theo}
	\label{T:stat}
	If {both}  (B\ref{I:B1}) and (B\ref{I:B2}) hold,
	then as $u\to \IF$
	\BQN 
p_{\vk b}(u)	\sim     T
\PICS
	u^{2/\alpha}\pk{ \vk X(0)> u \b},
	\label{es1}
	\EQN
	where $V_{\w} = \diag{\w} V \diag{\w}$
{and $\HA_{\alpha,V_{\w}}\in(0,\infty)$}.
Moreover,  \eqref{es1} holds for $T=T_u$ such that $\limit{u} {T_u}
=\IF$ and the right hand side of \eqref{es1} converges to 0 as $u\to \IF$.
\end{theo}

If $\w^{\top} V \w = 0$ in Assumption (B\ref{I:B2}), it is not clear whether Theorem~\ref{T:stat} still holds true. But in the following special case, we can obtain  the approximation of $p_{\b }(u)$  and give further an  explicit formula for the corresponding Pickands constant.

\begin{theo}
	\label{T:Vskew}
{Suppose that} $\alpha = 1$ and  Assumption (B\ref{I:B1}) is satisfied.
{If there exists} a $d \times d$ anti-symmetric matrix $V$ such that  
	\begin{equation}
	\label{e:B2'}
	\Sigma - \RR(t) \sim t V, \quad t \to 0
	\end{equation}
{and ${\left(V \w \right)_I \neq \vk{0}_I}$,}
	then \eqref{es1} holds with $ \HA_{\alpha,V_{\w}}$ substituted by $\sum_{1 \leq i \leq d} w_i |(V \w)_i|/2{>0}$ and {for $T=T_u$ as specified in \netheo{T:stat}.}

\end{theo}

\subsection{Non-stationary case}
We discuss next the case of non-stationary $\vk X$.  Let  for $t_0$, $t\in [0,T]$
$$\Sigma(t) =
R(t,t), \quad  \Sigma= \Sigma(\kk{t_0})$$
and  assume that $\Sigma$ is non-singular. As in the stationary case, for $\vk b \in \R^d \setminus (-\infty, 0]^d$  we set  $\vk w= \Sigma^{-1} \tilb$. Recall that
 $\tilb$ is the unique solution of  \eqref{e:QP} and
 $\vk w_I=\SIIIM \vk b_I> \vk 0_I$, $\vk w_J= \vk 0_J$ as mentioned in \eqref{keyWW}, where
  $I$, $J$ are defined with respect to $\Pi_{\Sigma}(\vk b)$; see Lemma~\ref{AL}. Next, for any $t\in [0,T]$ define
\BQN\label{vb}
\VB(t)=   \min_{\vk{z} \in
	[0,\IF)^d : \vk{z}^\top \vk{b}>0}  \frac{\vk z^\top \Sigma(t) \vk z}{(\vk z^\top \vk b)^2} =
\frac{1}{ \min_{\vk x \geq \vk b} \vk x^\top \Sigma^{-1}(t) \vk x},
\EQN
where the second equality holds under the assumption that $\Sigma(t)$ is non-singular, see \nelem{AL}. We shall refer to $\sigma^2_{\vk b}(t)$, $t\in [0, T]$  as the generalized variance function of $\vk X$.\\
\kk{For the 1-dimensional case ($d=1$)} it is known from several works of V.I. Piterbarg
that the local behaviour of the variance function around a unique \kk{maximizing point}
is crucial for the tail behaviour of supremum of non-stationary Gaussian processes.\\
In the vector-valued setting, the situation is more complex
since the
{local structure of the generalized variance function
in the neighbourhood of its maximizer is crucial.}
%
Therefore, {the following set of assumptions}
relates to both the covariance function and the generalized variance function of $\vk X$.   Namely, we shall assume that:

\begin{enumerate}[({D}1)]
	\item \label{I:D1} $\VB (t)$, $t\in [0,T]$ is continuous and attains its unique maximum  at
	$t_0 \in [0, T]$;
    \item\label{I:D2}  For all $t$ in $[0,T]$, there exists
   a continuous {$d\times d$ real matrix function} $A(t)$, $t\in [0,T]$ such that
    \BQN \label{eq:12}
    \Sigma(t) = A(t) A(t)^{\top}, \quad t\in  [0,T]
    \EQN
     and there  exist a $d\times d$ real matrix $\Xi$ and some $\beta > 0$ such that as $t\to t_0$
	\begin{equation} 	\label{e:At}
	A(t) = A(t_0) - |t - t_0|^{\beta} \Xi + o(|t-t_0|^{\beta}),
	\end{equation}
with
	\begin{equation}
	\label{e:bAXi}
	\tau_{\vk{w}} := \vk{w}^{\top} \Xi A(t_0)^{\top} \vk{w} {>0};
	\end{equation}
	\item\label{I:D3} There exist $\alpha \in (0, 2]$ and a $d\times d$ {real}
	matrix $D$
	such that  	for $t > s$
	\begin{equation}
	\label{e:Rts}
	R(t,s)  = A(t) \left( 
	\mathcal{I}_d - (t - s)^{\alpha} D +
	o(|t-s|^{\alpha}) \right) A(s)^{\top}  
	\end{equation}
	{as $t \to t_0, \ s \to t_0$;}
	\item\label{I:D4} There exist  $\gamma \in (0,2]$, $C\in (0,\IF)$ such that for all $s$, $t$ in an open  neighbourhood of $t_0$
	\begin{equation}
	\label{e:gamma}
	\E { \left|\vk{X}(t) - \vk{X}(s)\right|^2} \leq C \left| t - s \right|^{\gamma}.
	\end{equation}
\end{enumerate}

\begin{remark}
	\label{R:Vbt}
  \begin{enumerate}[i)]
  \item
    {As shown in Appendix,} \eqref{e:At} implies that 
    \begin{equation}
      \label{e:Vbmin}
      \VB(t_0) -\VB (t)  \sim  \frac{2
        \tau_{\vk{w}}}{(\tilb^\top \Sigma^{-1} \tilb)^2} |t-t_0|^{\beta} \quad \text{as } t \to t_0.
    \end{equation}
    {Thus if $\tau_{\w} > 0$, then  $\sigma_{\b}^2(t)$ has a local maximum point at $t = t_0$. Conversely, if $\sigma_{\b}^2(t)$ attains its maximum at $t = t_0$, then we have $\tau_{\w} \geq 0$. }

  \item By (D\ref{I:D2}), {Assumption}  (D\ref{I:D4}) follows if for some $\gamma>0$
    $$\normF{A(t)- A(s)}\le C \abs{t-s}^\gamma$$
    for all $s,t$ in an open neighbourhood of $t_0$.
  \end{enumerate}
\end{remark}

For a multivariate fBm  $\Y(t)$, $t\in \R $ with cmf $R_{\alpha, V}$ given by \eqref{e:KV}
{and a matrix $W$, }
 we introduce the multivariate  Piterbarg constant 
\begin{equation}
\label{e:pitK}
{\mathcal{P}_{\alpha, V,  W}} = \lim_{\Lambda \to \infty} \int_{\R^d} e^{\vk 1^\top \vk x}
\pk{\exists t\in {[0, \Lambda]}:  \Y_t - \big[\SA{t}{{V}} + \abs{t}^\alpha {W}\big]\id > \vk x}\, d\x,
\end{equation}
provided the limit exists. 

{The following theorem constitutes the main result of this section.
For compactness of the presentation we suppose that $t_0=0$
in (D1)-(D3); the other cases are commented in Remark \ref{Rem.2.5}.}
\begin{theo}
	\label{T:nonstationary}
Let $\vk{X}(t)$, $t \in [0,T] $ be a centered $\R^d$-valued Gaussian
	process satisfying  (D\ref{I:D1})--(D\ref{I:D4}) with $t_0=0$ and
	set 
  ${A=A(t_0)}$, $V = A D A^\top$, {$V_{\w} = \diag{\w} V \diag{\w}$ and $W_{\w} = \diag{\w} \Xi A^{\top} \diag{\w}$}.
	\begin{enumerate} 
		\item If $\beta > \alpha$ and ${\w^\top V \w > 0}$,
 then as $u\to \IF$
		\begin{equation}
		\label{e:asysub}
		p_{\vk b}(u) \sim
		 \HA_{\alpha,V_{\w}} \Gamma(1/\beta + 1) \tau_{\w}^{-\beta}
		u^{\frac{2}{\alpha}-\frac{2}{\beta}} \pk{ \vk X(t_0) > u \vk b},
		\end{equation}
		where
    $\HA_{\alpha, V_{\w}}{\in (0,\IF)}$. 
	\item If $\beta = \alpha$, then as $u\to \IF$
		\begin{equation}
		\label{e:critical}
		p_{\vk b}(u) \sim
		\mathcal{P}_{\alpha, V_{\w}, W_{\w}}  \pk{\vk X(t_0)> u \vk b} , 
		\end{equation}
		where $\mathcal{P}_{\alpha, V_{\w}, W_{\w}}   \in (0, \infty)$.
\kk{		\item If  $\beta < \alpha$,  then as $u\to\infty$
		\begin{equation}
		\label{e:asymsup}
	p_{\vk b}(u)
\sim
C_{\vk w}
\pk{ \vk X(t_0)> u \b},
   \end{equation}
}	
where 
$C_{\w} = 1 + {\tau_{\w}^{-1}} \sum_{i \in I} w_i \max(0, -(\Xi A^\top \w )_i)$.
\end{enumerate}
\end{theo}

\begin{remark}\label{Rem.2.5}
	\begin{enumerate}[i)]
		\item  		
{
	The constant $C_{\vk{w}}$ in \eqref{e:asymsup} equals 1
if and only if $(\Xi A^\top \w)_I \ge \vk{0}_I$.}
{Moreover, \eqref{e:asymsup} holds with the same constant if $t_0 \in (0,T]$.}

		\item {If $t_0 \in (0, T_0)$ or $t_0 = T$, then
Pickands constant $ \HA_{\alpha,V_{\w}} $ in (\ref{e:asysub})
has to be replaced by $ \HA_{\alpha,V_{\w}}+ \HA_{\alpha,V^\top_{\w}}$ or
$ \HA_{\alpha,V^\top_{\w}}$, respectively. Analogously,  Piterbarg constant is defined 
by \eqref{e:pitK} with $[0, \Lambda]$ replaced by $[- \Lambda, \Lambda]$ or $[-\Lambda, 0]$ if $t_0\in (0,T)$ or $t_0=T$, respectively.
}
	\end{enumerate}
	\end{remark}

\section{Examples}\label{s.examples}
In this section we apply the findings of Section \ref{s:extreme} to
three  important classes of vector-valued Gaussian processes, namely
operator  fO-U processes, {operator fBm's and their Lamperti transforms}.

\subsection{Operator  fO-U process}\label{s.OU}

Let $H$ be a symmetric matrix with all eigenvalues {$h_1$, $\ldots$, $h_d$} belonging to $(0, {1]}$ and consider {a}
centered stationary
{a.s.} continuous $\R^d$-valued Gaussian processes $\X(t)$, $t\ge 0$ {with cmf}
\[R(t,s)=e^{-|t-s|^{2H}}, \]
{where $t^H = \exp \left( H \log t \right)$ for $t > 0$. }
We call $\X$ \emph{an operator fO-U process}.
\\

 The existence of an fO-U process  follows from the fact that,
by the symmetry of $H$, we can write $H = Q \diag{h_1, \ldots, h_d} Q^\top$
for some
orthogonal matrix $Q$. Hence $\vk X(t)\stackrel{d}=Q\vk{Z}(t)$, $t\ge0$, where $Z_i(t)$, $t\ge0$, $i=1$, $\ldots$, $d$, are mutually independent stationary Gaussian processes
with covariance function
$r_i(t)=e^{-|t|^{2h_i}}$, respectively.
Consequently, setting $\RR(t)= R(t,0)$ we have
\[ \RR(t)=\mathcal{I}_d- |t|^{2 h_\star} Q\widetilde{\mathcal{I}}Q^\top+o(|t|^{2 h_\star}) \]
as $t\to 0$, where $h_{*} = \min_{1 \leq i \leq d} h_i$ and
\begin{eqnarray}\label{def.I}
\widetilde{\mathcal{I}}=\diag{e_1,...,e_d}, \  {\rm with} \  e_i=\left\{\begin{array}{cc}
                                                       0 & {\rm if}\ h_i>h_\star\\
                                                       1 & {\rm if}\ h_i=h_\star.
                                                     \end{array}\right.
\end{eqnarray}
{{Then (B2)} holds with $\alpha = 2 h_{*}$, $\Sigma = \mathcal{I}_d$ and $V = Q \widetilde{I} Q^{\top}$. Let $\tilb$ be the solution to the quadratic programming problem~\eqref{e:QP}, that is, $\tilb_i = b_i \vee 0$ for $1 \leq i \leq d$. }
	
In view of Theorem \ref{T:stat} we arrive at the following result:
\begin{prop}\label{p.OU}
Let  $\X(t)$, $t\in[0, T]$ be an $\R^d$-valued operator  fO-U process with
a symmetric matrix $H$ with all eigenvalues belonging to $(0, 1]$. {If $\tilb^{\top} Q \widetilde{\mathcal{I}} Q^{\top} \tilb > 0$,}  then
as $u\to \IF$
	\BQN 
p_{\vk b}(u)	\sim     T
	\mathcal{H}_{2h_\star,V_{{\tilb}}}
	u^{1/h_\star} \pk{ \vk X(0)> u \b},
\nonumber
	\EQN
	where $V_{{\tilb}} = \diag{{\tilb}} Q\widetilde{\mathcal{\mathcal{I}}}Q^\top \diag{{\tilb}}$.
\end{prop}

\subsection{Operator fBm}\label{s.fbm}
Let  $H$ be a $d \times d$ matrix and let $\X(t)$, $t\in \R$
be a centered, {a.s.} continuous $\R^d$-valued \emph{operator fBm} 
with index
$H$.
We shall assume that the following conditions hold:
\begin{enumerate}
\item [(O1)] There exists  an invertible matrix $Q$ such that
$H = Q U Q^{-1}$ with $U = \diag{h_1,\ldots, h_d}$ and $h_1$, $\ldots$, $h_d \in (0, 1{]}$;
\item [(O2)]  $\Sigma = \E { \X(1) \X(1)^{\top} }$ is non-singular and  $\X$ is time-reversible, i.e., $\E{ \X(t) \X(s)^{\top}} = \E { \X(s) \X(t)^{\top} }$ for all $t$ and $s$.
\end{enumerate}

We have that the {cmf of $\vk X$ is given by}
\[ R(t,\, s) =
 \frac{1}{2} \left( |t|^H \Sigma |t|^{H^\top} + |s|^H \Sigma |s|^{H^\top} - |t-s|^H \Sigma |t-s|^{H^\top} \right); \]
see for example \cite{DP11}.


For notational simplicity we shall suppose that $T=1$.
Write  $\Sigma = A^2$ for some $d\times d$  symmetric real-valued matrix  $A$.
{Then $\Sigma(t) = \E{\X(t) \X(t)^\top} = A(t) A(t)^\top$, where}
\begin{equation}
  \label{e:tHA}
  A(t) = t^H A = A(s) - s^{-1} (s-t) H A(s) + o(|s-t|). 
\end{equation}
{Let $\sigma_{\b}^2(t)$ be the generalized variance function of $\X$ defined by \eqref{vb}.
{Since in general the behaviour of $\sigma_{\b}^2(t)$ may be quite complex,
we focus on a tractable class, supposing that:}
  \begin{enumerate}
  \item [(O3)] The function $\sigma_{\b}^2(t)$, $t \in [0, 1]$ attains its unique maximum at $t = 1$.
  \end{enumerate}
}
\begin{remark}
    {Assumption (O3) holds if    $H \Sigma$ is positive definite in the sense that $\y^\top H \Sigma \y > 0$ for all $\y \neq \vk{0}$. Indeed, we have from \eqref{e:Vbmin} that
    \[
      \sigma_{\b}^2(s) - \sigma_{\b}^2(t) = \frac{2 \tau_{\w}(s)}{\left( \tilb(s)^\top \Sigma^{-1}(s) \tilb(s) \right)^2} \; (s - t) + o(|s-t|),
      \]
      where $\tilb(t)$ is the solution to the quadratic programming problem $\min_{\x \geq \b} \x^\top \Sigma^{-1}(t) \x$, $\w(s) = \Sigma^{-1}(s) \tilb(s)$ and
      \[
        \tau_{\w}(s) = s^{-1} \w(s)^\top H A(s) A(s)^\top  \w(s)
        = s^{-1} \w(s)^\top s^H H \Sigma \, s^{H^\top} \w(s).
      \]
      Since $H \Sigma$ is positive definite, we have that $\tau_{\w}(s) > 0$ for all $s > 0$ and hence $\sigma_{\b}^2(s)$ is strictly increasing in $s$.}      {Another condition to ensure Assumption (O3) is that $\b > \vk{0}$ and $H = \diag{h_1, \ldots, h_d}$ with $h_i \in (0, 1)$, $1 \leq i \leq d$. Under this setting, the cone $t^{-H} S_{\b}$ with  $S_{\b} = \left\{ \x:\ \x \geq \b \right\}$, is strictly increasing in $t \in (0, 1]$ and therefore
      \[
        \min_{\x \geq \b} \x^\top \Sigma^{-1}(t) \x = \min_{\y \in t^{-H} S_{\b}} \y^\top \Sigma^{-1} \y
      \]
 is strictly decreasing.
}
  \end{remark}

By \eqref{e:tHA} we have that \eqref{e:At} holds with
$\beta = 1$ and $\Xi = H A$.
Setting further
$h_{*} = \min_{1 \leq i \leq d} h_i$,  
 we have
\begin{eqnarray*}
\lefteqn{ A(t)^{-1} R(t,\, s) \left( A(s)^{-1} \right)^{\top}} \\
&=& \frac{1}{2} \left[ A \left( \frac{t}{s} \right)^{{H^{\top}}} A^{-1} + A^{-1} \left( \frac{s}{t} \right)^H  A
    - A^{-1} \left( \frac{|t-s|}{t} \right)^H \Sigma \left( \frac{|t-s|}{s} \right)^{H^\top} A^{-1} \right] \\
&=& \mathcal{I}_d - (t-s) D_1 {+ O(|t-s|^2)} - |t-s|^{2 h_{*}} D_2 + o(|t-s|^{2 h_{*}})
\end{eqnarray*}
as $t \uparrow 1$, $s \uparrow 1$ and $|t-s| \to 0$, where
$$D_1 = \frac{1}{2} \left( AH^{{\top}}A^{-1} - A^{-1} {H} A \right),
\quad D_2 = \frac{1}{2}  A^{-1} Q\widetilde{\mathcal{I}}Q^{-1} \Sigma
(Q\widetilde{\mathcal{I}}Q^{-1})^\top  A^{-1} ,$$
with $\widetilde{\mathcal{I}}$ given by \eqref{def.I}.
 If $D_1 = 0$, or equivalently $\Sigma {H^\top}  =  H \Sigma$,
 then Assumption (D\ref{I:D3}) holds with $\alpha = 2 h_{*}$ and $D = D_2$.
 If ${H^\top} \Sigma \neq \Sigma H$, then Assumption (D\ref{I:D3}) also holds for $h_{*} < 1/2$ with $\alpha = 2 h_{*}$ and $D = D_2$,  for $h_{*} = 1/2$ with $\alpha = 1$ and $D = D_1 + D_2$, whereas  for $h_{*} > 1/2$ with $\alpha = 1$ and $D = D_1$.
Note that $D_1$ is anti-symmetric \kk{and} hence $\w^{\top} A D_1 A^{\top} \w = 0$.


Applying Theorem~\ref{T:nonstationary}, we have the following asymptotics for operator fBm $\X$.

\begin{prop} \label{OPFB}
  Let $\X(t)$, $t\in[0, 1]$ be an operator fBm with index $H$.
Suppose that (O1)-(O{3}) hold
{ and	
  $\tau_{\w} = \tau_{\w}(1) = \w^{\top} H \Sigma \w > 0$.} 
	\begin{enumerate}[i)]
		\item If $h_{*} < 1/2$ and $\w^\top  {A D_2 A} \w > 0$, then  
		\[ 
		 p_{\vk b} (u) \sim {\mathcal{H}}_{2 h_{*}, V_{\w}}  {\tau_{\w}^{-1}} u^{\frac{1 - 2 h_{*}}{h_{*}}} \pk{ \X(1) > u \b }, \]
{with $V_{\w } = \diag{\w} A D_2 {A}\diag{\w}$.}
		\item If $h_{*} = 1/2$, then
		\[ 
				 p_{\vk b} (u)\sim {\sP_{1, V_{\w }, W_{\w}}} \pk{ \X(1) > u \b },
				  \]
       {with $V_{\w } = \diag{\w}A \left( D_1 + D_2 \right) A\diag{\w}$
       and $W_{\w}=\diag{\w} H \Sigma \diag{\w}$. }
		\item If $h_{*} > 1/2$, then
\[ 
  p_{\vk b} (u) \sim {C_{\w}}  \pk { \X(1) > u \b},
\]
{with $C_{\w} = 1 + \tau_{\w}^{-1} \sum_{i \in I} w_i \max \left( 0,\, - \left( H \tilb \right)_i \right)$.}
	\end{enumerate}
\end{prop}

\subsection{Lamperti transform of operator fBm's}\label{s.lamperti}
Let $\vk Y(t)$, $t\ge0$ be an operator fBm with index $H$. Suppose that (O1)-(O2)
hold and let
$\vk X(t)=(e^{-t})^H\Y(e^t)$, $t\ge 0$ be the Lamperti transform of {$\vk Y$}.
We follow the notation introduced in Section \ref{s.fbm}.
Clearly, $\mathbf{X}$ is stationary with
\begin{align*}
  \E{ \vk X(t) \vk X(s)^{\top}}
  =&
      \frac{1}{2}\left(           \Sigma e^{(t-s)H^\top}
      + e^{-(t-s)H}\Sigma -|1-e^{-(t-s)}|^H\Sigma |1-e^{t-s}|^{H^\top}
      \right)\\
  =&
      \Sigma- (t-s)\frac{1}{2}(H\Sigma-\Sigma H^\top) -
      |t-s|^{2h_\star}
      Q\widetilde{\mathcal{I}}Q^{-1} \Sigma
      (Q\widetilde{\mathcal{I}}Q^{-1})^\top
     +o(|t-s|^{2h_\star})
\end{align*}
for $t\ge s$, as $t-s\to 0$. 
Set 
$\tilde V= Q\widetilde{\mathcal{I}}Q^{-1} \Sigma (Q\widetilde{\mathcal{I}}Q^{-1})^\top$.
{Recall that $\tilb$ solves the quadratic programming problem~\eqref{e:QP} and we set $\w = \Sigma^{-1} \tilb$. }
%
 {Applying Theorem~\ref{T:stat} and ~\ref{T:Vskew}, we have the following proposition.
\begin{prop}\label{p.LAMP}
{Let $\X$ be the Lamperti transform of an
operator fBm with index $H$ satisfying (O1)-(O2).}
  \begin{enumerate}[i)]
  \item Assume $H\Sigma=\Sigma H^\top$ or $H\Sigma\neq\Sigma H^\top$ but $h_*<1/2$. If $\w^{\top} \widetilde{V} \w > 0$, then as $u \to \infty$,
    \begin{equation}
      \label{e:lamp}
      p_{\vk{b}} (u) \sim T \, \HH_{\alpha, V_{\w}} u^{2/\alpha} \, \pk{\X(0) > u \b} ,
    \end{equation}
    with $\alpha=2h_*$ and $V=\widetilde{V}$.
  \item Assume $H\Sigma\neq\Sigma H^\top$ and $h_*=1/2$. If $\w^{\top} \widetilde{V} \w > 0$, then \eqref{e:lamp} holds with $\alpha = 1$ and $V = H\Sigma-\Sigma H^\top + \widetilde{V}$.
  \item Assume $H\Sigma\neq\Sigma H^\top$ and $h_* > 1/2$. Set $\alpha = 1$ and $V = H\Sigma - \Sigma H^{\top}$. If $\left(V \w \right)_I \neq \vk{0}_I$, then \eqref{e:lamp} holds with $\HH_{\alpha, V_{\w}}$ replaced by $\sum_{1 \leq i \leq d} w_i \left| \left( V \w \right)_i \right| / 2$.
  \end{enumerate}
\end{prop}
}

\section{\kk{Auxiliary} Results}\label{s.aux}
In this section we include some key tools for vector-valued Gaussian processes, which will be used
in the proofs of the main results and are of some interest on their own right.
\kk{We postpone all the proofs of lemmas presented in this section to Appendix.}
We explain first the properties of the solution of $\Pi_{\Sigma}(\b)$ followed by a section on uniform approximation of tails of functionals of Gaussian processes.

\def\Si{\SIIIM}
\subsection{Quadratic programming problem}
For a given non-singular $d\times d$ {real} matrix $\Sigma$ we consider  the quadratic programming problem
\begin{equation}
\Pi_\Sigma(\b): \text{minimise $ \x^\top \Sigma^{-1} \x $ under the linear constraint } \x \ge \b.
\end{equation}
Below $J= \{  1 \ldot d\}\setminus I$ can be empty; the claim in \eqref{empt} is formulated under the assumption that $J$ is non-empty. \\
\BEL \label{AL}  Let $d \geq 2$ and
$\Sigma$ a $d \times d$ symmetric positive definite  matrix with inverse $\SIM$. If  $\b \inr^d \setminus (-\IF, 0]^d $, then
$\Pi_{\Sigma}(\vk b)$ has a unique solution $\tilb$ and there exists a unique non-empty
index set $\IB\subset \{1 \ldot d\}$ with $m\le d$ elements such  that
\BQN \label{eq:IJ}
\tilb_{\IB}&=&
\b_{\IB} \not=\vk{0}_I, \\
\label{empt}
\tilb_{J}&=& \SIJI \SIIIM \b_{\IB}\ge \b_{\JB}, \quad \SIIIM \b_{\IB}>\vk{0}_I,\\
\label{eq:alfa} \min_{\x \ge  \b}\x^\top \SIM\x&=& \tilb ^\top \SIM \tilb  =  \b_{\IB}^\top \SIIIM\b_{\IB}>0,\\
\label{eq:alfaB}
\max_{ \vk{z}\in [0,\IF)^d: \vk{z}^\top \b >0} \frac{(\vk{z}^\top \b)^2 }{\vk{z}^\top \SI \vk{z}} &=&
\frac{(\vk{w}^\top \b)^2 }{\vk{w}^\top \SI \vk{w}}=\min_{\x \ge  \b}\x^\top \SIM\x,
\EQN
with $\vk{w}= \SIM \tilb$ satisfying $\vk w_I= \SIIIM \b_I> \vk 0_I, \vk w_J= \vk 0_J$.
\EEL

Define the solution map of the quadratic programming problem \eqref{e:QP} by
$\mathcal{P}:\Sigma^{-1} \mapsto \tilb$ with $\tilb$ the unique solution to $\Pi_\Sigma(\b)$.
The next result  is a special case of \cite[Thm 3.1]{Hag79}.

\begin{lem}
	\label{L:lip}
	$\mathcal{P}$ is Lipschitz continuous on compact subset of the space of {real} $d \times d$ symmetric  positive definite matrices.
\end{lem}

\subsection{Uniform tail appropximation for functionals of families of Gaussian processes} \label{sec:weakU}
A key role  in the analysis of {extremes} of Gaussian processes is played by the continuous mapping theorem,
the idea appeared first in \cite{Pit72, MR0348906} and it is used extensively in the monographs \cite{Pit96,Pit20}.\\
Our main tool that shall compensate for  the lack of Slepian  lemma  is the uniform approximation of the tail of supremum
of threshold-dependent Gaussian processes. We present below a general result 
where the {tail distributions for}  a continuous functional of a family of Gaussian processes are uniformly approximated.

Let  $\{\xxiu(\t)$, $\t\in\bD\}$, $u>0$, $\tau \in Q_u \subset \R$ be a family of centered, $d$-dimensional  vector-valued Gaussian processes with a.s. continuous sample paths {and parameter space $E $ which is assumed to be a compact subset of $\R^k$. For notational simplicity we discuss below only case $k=1$.} Denote its cmf  by  $R_{u, \tau}(t,s) = \E{\xxiu(t) \xxiu(s)^\top}$ and let  $C(E)$ be  the separable  Banach space  of all $\R^d$-valued continuous functions on $E$ equipped with the sup-norm and assume for simplicity  {that the origin $0$ of $\R^k$ belongs to $E$.}

\BEL \label{lemGausUnif}
Suppose that 
$\xxiu(0)=\vk{0}$ almost surely and $\vk{Y}(t)$, $t\in E$ is a  Gaussian  processes with a.s. continuous sample paths.
Let $\vk{f}_{u,\tau}(t)$, $\vk f(t)$, $t\in E$  be deterministic $\R^d$-valued continuous functions.
Assume  that
\BQN\label{unifF}
\limit{u} \sup_{t \in \bD, \tau \in Q_u} \abs{\vk{f}_{u,\tau}(t) - \vk{f}(t)} = 0
\EQN
and
\BQN\label{covcon}
\limit{u} \sup_{t, s \in \bD, \tau \in Q_u} \normF{ R_{u,\tau}(t,s) - \E{ \Y(t) \Y(s)^\top}} = 0.
\EQN
If further for some  $C\in (0,\IF)$, $\gamma\in (0,2]$ and any $s$, $t\in E$
\BQN \label{uniftauT}
\limsup_{u\to \IF} \max_{1 \le i\le d}\sup_{\tau \in Q_u}\E{[X_{i,u,\tau}(t)- X_{i,u,\tau}(s)]^2} \le C\abs{t-s}^\gamma,
\EQN
then for any continuous functional $\Gamma:
C(E) \mapsto \R$
\BQN \limit{u} \sup_{ \tau \in Q_u} \Abs{ \pk{ \Gamma(\xxiu- \vk{f}_{u,\tau})> s} - \pk{\Gamma(\vk{Y}- \vk{f})>s}}=0
\EQN
is valid  for all  continuity point $s$ of the distribution function of $\Gamma(\vk{Y}- \vk{f})$.
\EEL

Application of \nelem{covcon}  requires the determination of the  continuity points of the functional $\Gamma(\vk{Y}- \vk{f})$. The next result is useful in that context.

{\BEL\label{l.cont}
If  $\vk{Y}(t)$, $t\in E$ is a  Gaussian  process with a.s. bounded sample paths, then
$
\pk{\exists
		t \in E: \ \vk Y (t) > \x}
$
is continuous on $\R^d$, except  at most at points
of Lebesgue measure $0$ in $\R^d$ belonging to
$\bigcup_{i=1}^d \R^{i}\times \{s_i\}\times \R^{d-i-1}$,
where
$s_i=\inf\{s:\pk{\sup_{t\in E} Y_i(t)\le s}>0\}$ for $i=1,...,d$.
\EEL}

\subsection{Borell-TiS \& Piterbarg  inequalities}\label{s.BTP}


The Borell-TIS inequality, see e.g. \cite{AdlerTaylor}, is 
very useful and crucial in numerous theoretical problems.
Under some weak {assumptions}, its {refinement i.e.,  Piterbarg} inequality (\cite[Thm~8.1]{Pit2001})
gives a more precise upper bound  for supremum of Gaussian random fields.
{In the following lemma we present an
extension of this tools to vector-valued setting}.


\BEL\label{GBorell} 
Let $\vk{Z}(t),\  t\in\VV$ be a separable centered $d$-dimensional vector-valued Gaussian process
having components  with  a.s.\ continuous
trajectories. {Assume that $\Sigma(t) = \E{\Z(t)\Z(t)^\top}$ is
non-singular for all $t \in E$.}  Let $\vk{b} \in \R^d\setminus (-\IF, 0]^d$ and define
$\VB{(t)}$ as in \eqref{vb}.
If ${\VB =}\sup_{t\in \VV} \VB(t) \in (0,\IF)$, then
there exists some positive constant $\mu$ such that for all $u>\mu$
\BQN \label{eqBO}
\pk{\exists {t\in\VV} :\ \vk{Z}(t)>u\vk{b}}\le
\exp\LT(-\frac{(u-\mu)^2}{2 \sigma_{\vk b} ^2}\RT).
\EQN
If further {for some}  $C \in (0,\IF)$ and $\gamma \in (0,2]$
\begin{equation}
\label{eq:ZG2}
\sum_{1 \le i\le k} \E{(Z_i(t)-Z_i(s))^2}\le C \abs{t-s}^{\gamma}
\end{equation}
and
\begin{equation}
\label{e:holder}
\normF{ \Sigma^{-1}(t) - \Sigma^{-1}(s) } \leq C  \abs{t-s}^{\gamma}
\end{equation}
hold for all $t$, $s\in\VV$, then  for all $u$ positive
\BQN\label{IPIT}
\pk{\exists {t\in\VV} : \ \vk{Z}(t)>u\vk{b}}
\le C_* \mathrm{mes}(\VV) \ u^{ \frac{2 d}{\gamma}-1} \exp\LT(-\frac{u^2}{2 \sigma^2_{\vk b}}\RT),
\EQN
where $C_*$ is some positive constant not depending on $u$.

In particular, if  $\sigma_{\vk{b}}(t)$, $t\in  \VV$ is continuous and
achieves its unique maximum at some fixed point $t_0\in \VV$, then \eqref{IPIT}
is still valid if \eqref{eq:ZG2} and \eqref{e:holder} are assumed to hold only for  all $s$, $t\in E$ in an open  neighbourhood of $t_0$.
\EEL

\subsection{Uniform approximation on short intervals (Pickands-Piterbarg Lemma)}
\label{s:pickands}
{Following notation introduced in Section \ref{sec:weakU}}  {denote by  $R_{u,\tau}(\cdot,\cdot)$ } the cmf  of $\xxiu$.
We shall impose  the following assumptions:

\begin{enumerate}[({A}1)]
	\item \label{I:A1} For all large $u$ and all $\tau \in Q_u$ the matrix $\Sigma_{u,\tau} = R_{u, \tau}(0,0)$ is
	positive definite and
	\begin{equation}
	\label{e:Sigmau}
	\lim_{u \to \infty} u \sup_{\tau \in Q_u}\normF{ \Sigma - \Sigma_{u,\tau} } = 0
	\end{equation}
	holds for some positive definite matrix $\Sigma$;
 	\item\label{I:A2} There exist a continuous $\R^d$-valued function
	$\vk{d}(t), t\in \bD$ and a  continuous
	matrix-valued function $K(t,s)$, $(t,s) \in \bD \times \bD$ such that
	\begin{gather}
	\lim_{u \to \infty} \sup_{ \tau \in Q_u, t \in \bD} u \normF{
		\Sigma_{u, \tau} - R_{u,\tau}(t,0) } =
	0,     \label{e:Rt0Sigma} \\
	\lim_{u \to \infty} \sup_{ \tau \in Q_u, t \in \bD} \Abs{ u^2 \left[  \Sigma_{u,\tau} - R_{u,\tau}(t, 0) \right]
		\Sigma^{-1} \tilb - \vk{d}(t)}=0     \label{e:dtA}
	\end{gather}
	and
	\begin{equation}
	\label{e:Kts}
	\lim_{u \to \infty} \sup_{ \tau \in Q_u, s,t \in \bD}
	\normF{  u^2 \big[ R_{u, \tau}(t,s) - R_{u, \tau}(t,0)
		\Sigma_{u,\tau}^{-1} R_{u, \tau}(0,s) \big]  - K(t,s)} =0;
	\end{equation}
	\item \label{I:A3} There exist positive constants $C$ and $\gamma \in (0,2]$ such that for any $s$, $t\in E$ 
	\begin{equation}
	\label{e:holderX}
	\sup_{\tau \in Q_u} u^2 \E{\left| \xxiu(t) - \xxiu(s) \right|^2} \leq C \abs{ t-s}^\gamma.
	\end{equation}
\end{enumerate}

\begin{remark}
	\begin{enumerate}[(i)]
		\item The existence of $\Sigma_{u,\tau}^{-1}$ follows from the positive definiteness of $\Sigma$ and condition  \eqref{e:Sigmau}. Further since
		\begin{align*}
		& R_{u, \tau}(t,s) - R_{u, \tau}(t,0)
		\Sigma_{u,\tau}^{-1} R_{u, \tau}(0,s) \\ = & \E{\left[ \X_{u, \tau}(t)
			- R_{u, \tau}(t, 0) \Sigma_{u, \tau}^{-1} \X_{u, \tau}(0)
			\right] \left[ \X_{u, \tau}(s) - R_{u, \tau}(s, 0) \Sigma_{u,
				\tau}^{-1} \X_{u, \tau}(0) \right]^{\top} },
		\end{align*}
		\kk{then} $K(t,s)$ is a matrix-valued non-negative definite
		function on $\bD \times \bD$ with $K(t,s) = K(s,t)^{\top}$. Consequently, for $K_{\w}(t,s) = \diag{\w} K(t,s) \diag{\w}$ with $\vk w$ some vector in $\R^d {\setminus \{\vk 0\}}$  there exists a centered, $\R^d$-valued
		Gaussian random field $\Y(t)$, $t \in \bD$ with $\Y(0) = 0$ and cmf  $K_{\w}$.
		\item If for some continuous matrix-valued function $V(t,s)\in {\R^{d\times d}}$, $(t,s)\in \bD \times \bD$
		\begin{equation}
		\label{e:Vts}
		\limit{u}  {\sup_{\tau \in Q_u, s,t\in E}} \Abs{ u^2 \Bigl[ \Sigma_{u,\tau} - R_{u,\tau}(t,s) \Bigr] - V(t,s)}=0,
		\end{equation}
		then (A\ref{I:A2}) holds with $\vk{d}(t) = V(t,0) \w$ and $K(t,s) = V(t,0) + V(0, s) - V(t,s)$.
		\item Let $A_{u, \tau}(t)$ (resp. $A_{u, \tau}$) be the square roots of
		the positive definite matrices $\Sigma_{u, \tau}(t)$ (resp. $\Sigma_{u,
			\tau}$).
		Note that
		\[ R_{u,\tau}(t,s) = A_{u,\tau}(t) C_{u,\tau}(t,s)
		A_{u,\tau}(s)^{\top} ,
		\]
		with
		$C_{u,\tau}(t,s)$ the  cmf  of
		$A_{u,\tau}(t)^{-1} \X_{u,\tau}(t)$. Under the condition that $\lim_{u \to \infty} A_{u,\tau}(t) = A$ uniformly in $t \in E$ and $\tau \in Q_u$,
		the convergence in \eqref{e:Kts} only depends on the correlation
		structure of $\X_{u,\tau}(t)$, that is, if we {suppose that}
		\begin{equation}
		\label{e:limCts}
		\lim_{u \to \infty} {\sup_{\tau \in Q_u, s,t\in E}} \Abs{  u^2 \left[ C_{u,\tau}(t,s) - C_{u,\tau}(t,0)
		C_{u,\tau}(0,s) \right] - \widetilde{K}(t,s)}=0,
		\end{equation}
		then \eqref{e:Kts}
		holds with $K(t,s) = A \widetilde{K}(t,s) A^{\top}$.
	\end{enumerate}
\end{remark}

{For $\Y(t)$, $t \in \bD$  a centered $\R^d$-valued Gaussian
process with a.s.\ continuous sample paths with  cmf $K(s,t)$, $s$, $t\in \bD$ and  an $\R^d$-valued function $\vk d$ define below
\BQN
 H_{\Y,\d}(E)=  \int_{\R^d} e^{\id^{\top} \x} \pk{\exists t \in E:\ \Y(t) - \vk d (t)  > \x} d \x.\label{def.H}
\EQN
}
\BEL\label{Lem1}
{Suppose that $\xxiu(t)$, $t\in \bD$, $u > 0$, $\tau \in Q_u$ satisfy  (A\ref{I:A1})-(A\ref{I:A3}).
Let $\w = \Sigma^{-1} \tilb$ where $\tilb$ is the unique solution of
$\Pi_{\Sigma}(\b)$.}
If   $\Y(t)$, $t \in \bD$ has cmf 
$R(t,s)=\diag{\w} K(t,s) \diag{\w}$ 
and $\d_{\w}(t) = \diag{\w} \d(t)$,  {then we have}
\begin{equation}
\label{eq:lem1}
\limit{u} \sup_{\tau \in Q_u}\ABs{  \frac{ \pk{\exists {t \in \bD}:\ \xxiu(t)> u \vk{b} }}
	{\pk{\X_{u,\tau}(0) > u \b}}
	- \H_{\Y,\d_{\w}}(E)
}=0.
\end{equation}
\EEL

\BRM \label{remA}
 If we suppose stronger assumptions on $\Sigma_{u,\tau}$, for instance
$$\lim_{u \to \infty} \sup_{t\in Q_u} \normF{ u^2 \left[ \Sigma - \Sigma_{u,\tau} \right]  - \Xi} =0,$$ then as $u\to \IF$
\[
\pk{\X_{u,\tau}(0) > u \b}
\sim
e^{-\w^{\top} \Xi \w / 2}
\pk{{\vk{\mathcal{N}}} > u \b},
\]
where ${\vk{\mathcal{N}}}$ is a centered Gaussian vector with  covariance matrix $\Sigma$.
\ERM

\section{Proofs of Main Results}\label{Secproofs}

For notation simplicity, unless otherwise defined, $\vk w$ in the following is defined as $\w = \Sigma^{-1} \tilb$.
Recall that $I$ with $m$ elements is the index set related to $\Pi_{\Sigma}(\b)$ with unique solution  $\tilb$ where $\vk b\inr^d$ is assumed to have  at least one positive component.
Further by \eqref{keyWW} and Assumption (B\ref{I:B2}) we have that   $\vk w_J= \vk 0_J$ and
\BQN \label{tauw2}
\xi_{\w}= \vk w^\top V \vk w= \vk w_I^\top V_{II}  \vk w_I> 0.
\EQN
{Otherwise specified, in the following  $\Y(t)$, $t\inr$ is a {centered $\R^d$-valued} Gaussian process with cmf $\diag{\w} R_{\alpha,V}(t,s) \diag{\w}$ where
	$R_{\alpha,V}$ is defined in \eqref{e:KV}.
}
\subsection{Proof of \netheo{T:stat}}
Let  $\vk X(t)$, $t\in [0,T]$ be  a stationary centered $\R^d$-valued Gaussian process.
{Before proceeding to the proof of Theorem \ref{T:stat}, we shall derive
some useful asymptotic bounds.} The first
lemma is crucial for the  negligibility of the double-sum. Below we set $\Delta(\lambda, \Lambda) = [0, \Lambda] \times [\lambda, \lambda + \Lambda]$
and
\begin{equation}
  \label{e:omegab}
  P_{\vk b}(\lambda, \Lambda, u) =   	\pk{\exists (t,s) \in u^{-2/\alpha} \Delta(\lambda, \Lambda)  :\ \X(t) > u \b,\ \X(s) > u \b}.
\end{equation}

\begin{lem}
	\label{L:double} If Assumption (B\ref{I:B2}) holds, then there exist
	positive constants $C$, {$\varepsilon$} and $n_0$ such that for every $\lambda \geq
	n_0 \Lambda >
	0$ with 	{$\lambda + \Lambda < \varepsilon u^{2/\alpha}$}
	\begin{equation}
	\label{e:double}
	P_{\vk b}(\lambda, \Lambda, u) \leq
	C\, \pk{\vk X(0)> u \vk b}  e^{  -
		\frac{\lambda^{\alpha}}{{16}} \xi_{\w} }.
	\end{equation}
\end{lem}

\prooflem{L:double}
Since
\begin{align*}
P_{\vk b}(\lambda, \Lambda, u) \leq
\pk{\exists (t,s) \in u^{-2/\alpha} \Delta(\lambda,\Lambda):\ \X_I(t) > u \b,\, \X_I(s) > u \b}
\end{align*}
and $\pk{\X(0) > u \b} \sim c_1 \pk{\X_I(0) > u \b_I}$ as $u \to \infty$, it suffices to prove \eqref{e:double}
with $\X$ replaced by $\X_I$. In the rest of the proof, we assume for notational  simplicity  that
$I = \{1,\ldots, d\}$ {and consequently $\tilb = \b$}. Set below
$$V^+ = \frac{1}{2} (V + V^{\top}), \quad V(t) ={\SA{t}{V}} =|t|^{\alpha} \left( V \id_{\{t \geq 0\}} + V^{\top} \id_{\{t<0\}} \right).$$
By Assumption (B\ref{I:B2}), for every $\delta > 0$ there exists $\varepsilon >
0$ such that for every $|t| < \varepsilon$ we have
\begin{equation}
\label{e:variance}
\left\Vert \Sigma - \RR(t) - V(t) \right\Vert_{\rmF} \leq \delta |t|^{\alpha}.
\end{equation}
Set $R_u(t) = \RR(u^{-2/\alpha} t)$ and
define $\vk{X}_u(t,s) = \frac{1}{2} \left( \vk{X}(u^{-2/\alpha} t) +
\vk{X}(u^{-2/\alpha} s) \right)$ which has cmf
\begin{align*}
R_u(t,s;t_1,s_1) =
& \E{\vk{X}_u(t,s) \vk{X}_u(t_1,s_1)^{\top}} \\
= & \frac{1}{4} \left( R_u(t-t_1) + R_u(s-s_1) + R_u(t-s_1) + R_u(s-t_1) \right).
\end{align*}
Further set
$$\Sigma_{u,\lambda} = \E{\vk{X}_u(0,\lambda)
	\vk{X}_u(0,\lambda)^{\top}} = \frac{1}{4} \left[ 2 \Sigma + R_u(\lambda) + R_u(-\lambda) \right]$$
 and
\[ V(t,s;t_1,s_1) = V(t-t_1) + V(s-s_1) + V(t-s_1) + V(s-t_1) - V(\lambda) - V(-\lambda). \]
In view of  \eqref{e:variance}  we have for $\lambda \in (0, \varepsilon u^{2/\alpha})$ that
\begin{equation}
\label{e:SigmauB}
\left\Vert u^2 (\Sigma - \Sigma_{u,\lambda}) - \frac{\lambda^{\alpha}}{2} V^+\right\Vert_{\rmF} \leq \frac{1}{2} \delta \lambda^{\alpha}
\end{equation}
and
\begin{equation}
\label{e:SigmaRu}
\left\Vert u^2 \big( \Sigma_{u,\lambda} - R_u(t,s;t_1,s_1) \big) - \frac{1}{4} V(t,s;t_1,s_1) \right\Vert_{\mathrm{F}} \leq \frac{3}{2} \delta \lambda^{\alpha}.
\end{equation}
Since $\Sigma_{u,\lambda}^{-1} - \Sigma^{-1} = \Sigma_{u,\lambda}^{-1} \left( \Sigma - \Sigma_{u,\lambda} \right) \Sigma^{-1}$,
we have from \eqref{e:SigmauB} that
\BQN
\label{e:Sigmau1}
\normF{\Sigma_{u,\lambda}^{-1} - \Sigma^{-1}} & =& O(\lambda^{\alpha} u^{-2})
\EQN
and
\BQN
\label{e:SigmauSigma}
\begin{aligned}
& \normF{ u^2 \left( \Sigma_{u,\lambda}^{-1} - \Sigma^{-1} \right) - \frac{\lambda^\alpha}{2} \Sigma^{-1} V^+ \Sigma^{-1}} \\
\leq &
\normF{\Sigma^{-1} \left[ u^2 \left( \Sigma - \Sigma_{u,\lambda} \right) - \frac{\lambda^{\alpha}}{2} V^+ \right] \Sigma^{-1}}
+
\normF{ u^2 \left( \Sigma_{u,\lambda}^{-1} - \Sigma^{-1} \right) \left( \Sigma - \Sigma_{u,\lambda} \right) \Sigma^{-1}} \\
\leq &
\delta \lambda^{\alpha} \normF{\Sigma^{-1}}^{{2}}.
\end{aligned}
\EQN
Therefore, for $\delta>0$ sufficiently small
\begin{equation}
\label{e:Sigmau2}
u^2 \b^{\top} \Sigma_{u,\lambda}^{-1} \b \geq u^2 \b^{\top} \Sigma^{-1} \b + \frac{\lambda^{\alpha}}{4} \xi_{\w}.
\end{equation}

Conditioning on $\vk{X}_{u}(0,\lambda) = u \vk{b} -
u^{-1} \vk{x}=:a_u(\vk x)$ we obtain further
\begin{align*}
P_{\vk b}(\lambda, \Lambda, u)
\leq & \pk{\exists (t,s) \in \Delta(\lambda,\Lambda) :\
	\vk{X}_u(t,s) > u \vk{b}} \\
= & u^{-d} \int_{\R^d} \pk{\exists (t,s) \in \Delta(\lambda,
	\Lambda) :\ \vk{X}_u(t,s) > u \vk{b} \,\big|\,
	\vk{X}_{u}(0,\lambda) = a_u(\vk x)}
\varphi_{\Sigma_{u,\lambda}}(a_u(\vk x)) d
\vk{x} \\
= & u^{-d} \int_{\R^d} {J}_u(\x)
\varphi_{\Sigma_{u,\lambda}}(a_u(\vk x)) d \vk{x},
\end{align*}
where
$$ {J}_u(\x)=\pk{\exists (t,s) \in \Delta(\lambda,
	\Lambda) :\ \vk{\chi}_u(t,s) > \vk{x}}$$
and $\vk{\chi}_u(t,s)$ is the conditioned process
$u(\vk{X}_{u}(t,s) - u \vk{b}) + \vk{x}$ given
$\vk{X}_{u}(0,\lambda) = a_u(\vk x)$. By \eqref{e:Sigmau1} and \eqref{e:Sigmau2}
\begin{equation}
\label{e:phiSigmau}
\begin{aligned}
\varphi_{\Sigma_{u,\lambda}}(a_u(\vk x))
\leq & \varphi_{\Sigma_{u,\lambda}}(u \vk{b}) \exp \left(
\vk{b}^{\top} \left( \Sigma_{u,\lambda} \right)^{-1}
\vk{x} \right) \\
{\leq} & \varphi_{\Sigma}(u \vk{b}) \exp \left( -
\frac{\lambda^{\alpha} \xi_{\w}}{{8}}
\right)
e^{\left(\vk{w} + {O(\lambda^\alpha u^{-2})}\right)^{\top} \vk{x} }.
\end{aligned}
\end{equation}
Consequently,
for all $u$ large enough
\begin{equation}
\label{e:dsconditioned}
\begin{aligned}
P_{\vk b}(\lambda, \Lambda, u)
  \leq  2 u^{-d} \varphi_{\Sigma}(u \vk{b}) \exp \left( -
\frac{\lambda^{\alpha}}{{8}} \xi_{\w}
\right) \int_{\R^d} e^{\left( \vk{w} + {O(\lambda^\alpha u^{-2})}
	\right)^{\top} \vk{x}} P_u(\x)
d \vk{x}.
\end{aligned}
\end{equation}

Given   $F \subset \left\{ 1,\ldots, d \right\}$ let
$\Omega_F = \left\{ \x \in \R^d:\ x_i > 0,\, x_j < 0,\, i \in F,\, j \not\in F \right\}$. If $F$ is empty i.e., $F = \emptyset$, then
\begin{equation}
\label{e:Fempty}
\int_{\Omega_{\emptyset}} e^{\left( \w + O(\lambda^{\alpha} u^{-2}) \right)^{\top} \x} {J}_u(\x)d \x
\leq
\int_{\Omega_{\emptyset}} e^{\left( \w + O(\lambda^{\alpha} u^{-2}) \right)^{\top} \x} d \x
\leq
\frac{2}{\prod_{1 \leq i \leq d}w_i}.
\end{equation}
Assume next that $F \neq \emptyset$ and define for $u>0$
$$d_u(t,s) = - \E{\vk{\chi}_u(t,s)}, \quad \eta_u(t,s) = \w_F^{\top} \left( \vk{\chi}_{u,F}(t,s) + \d_{u,F}(t,s) \right).$$
For  all $\x \in \Omega_F$
\begin{equation}
\label{e:dsupF}
P_u(\x) \leq \pk{\exists (t,s) \in \Delta(\lambda,\Lambda):\ \eta_u(t,s) > \w_F^{\top} \x_F - \w_F^{\top} \d_{u,F}(t,s)}.
\end{equation}

Since  $\X_u(0,\lambda)$ is independent of $\X_u(t,s) - R_u(t,s;0,\lambda) \Sigma_{u,\lambda}^{-1} \X_u(0,\lambda)$ we obtain
\BQNY
 d_u(t,s) &=& u^2 \left[\Sigma_{u,\lambda} - R_u(t,s;0,\lambda) \right] \Sigma_{u,\lambda}^{-1} \b
- \left[ \Sigma_{u,\lambda} - R_u(t,s;0,\lambda) \right] \Sigma_{u,\lambda}^{-1} \x
\EQNY
and the cmf of $\vk{\chi}_u$ is
\[ K_u(t,s;t_1,s_1) = u^2 \left[ R_u(t,s;t_1,s_1) - R_u(t,s;0,\lambda) \Sigma_{u,\lambda}^{-1} R_u(0,\lambda;t_1,s_1) \right].
\]

Set
\[ \d(t,s) = \frac{1}{4} \left[ t^{\alpha} + (s-\lambda)^{\alpha} + s^{\alpha} - \lambda^{\alpha} \right] V \w - \frac{1}{4} \left( \lambda^{\alpha} - (\lambda-t)^{\alpha} \right) V^{\top} \w
\]
and
\begin{align*}
 K(t,s;t_1,s_1) =& \frac{1}{4} \left( V(t,s;0,\lambda) + V(0,\lambda; t_1, s_1) - V(t,s;t_1,s_1) \right) \\
  =& \frac{1}{4} \left[ \psi(t,s;t_1,s_1) V + \psi(t_1,s_1;t,s) V^\top - V(t-t_1) - V(s-s_1) \right],
\end{align*}
where
$$\psi(t,s;t_1,s_1) = t^{\alpha} + (s-\lambda)^{\alpha} + s^{\alpha} - \lambda^{\alpha} +(\lambda-t_1)^{\alpha} - (s-t_1)^{\alpha}.$$
By \eqref{e:SigmaRu} and \eqref{e:SigmauSigma}, there exists $c_2 > 0$ such that for $\lambda + \Lambda < \varepsilon u^{2/\alpha}$ and $(t,s)$, $(t_1,s_1)\in \Delta(\lambda,\Lambda)$
\begin{equation}
\label{e:echiu}
\left| \vk{d}_u(t,s) - \d(t,s)  - \left[ \Sigma_{u,\lambda} - R_{u}(t,s;0,\lambda) \right]
\Sigma_{u,\lambda}^{-1} \vk{x} \right| \leq c_2 \delta \lambda^{\alpha}
\end{equation}
holds and further
\begin{equation}
\label{e:vchiu}
\normF{K_u(t,s;t_1,s_1) - K(t,s;t_1,s_1)} \leq c_2 \delta \lambda^{\alpha}.
\end{equation}
Using the inequality
\bqn{\label{aero}
	\left| p^{h} - q^{h} \right| \le h \max( p^{h-1} , q^{h-1} ) \left| p-q \right|
}
valid for all $p, q$ and { $h$ positive}, we obtain that
\[
\left| \d(t,s) \right| \leq c_3 \lambda^{\alpha-1} \Lambda \quad \text{and }\normF{K(t,s;t_1,s_1)} \leq c_3 \lambda^{\alpha-1} \Lambda
\]
holds for some positive constant $c_3$. Therefore, we can choose $n_0$ large enough so that, for $\lambda \geq n_0 \Lambda$ and $\lambda + \Lambda < \varepsilon u^{2/\alpha}$
\begin{equation}
\label{e:duts}
\inf_{(t,s) \in \Delta(\lambda,\Lambda)} \left[ - \w_F^{\top} \d_{u,F}(t,s) \right] \geq - \frac{1}{2} \w_F^{\top} \x_F - c_4 \delta \lambda^{\alpha}
\end{equation}
and 	
{\begin{equation}
    \label{e:etauts}
    \sigma_F^2 = \sup_{(t,s) \in \Delta(\lambda, \Lambda)} \Var \left( \eta_u(t,s) \right) \leq c_4 \delta \lambda^\alpha.
  \end{equation}
}

Since the conditional variance is always less than or equal to the unconditional
one, we have that
\begin{align*}
& \Var \left( \eta_u(t,s) - \eta_u(t_1,s_1) \right) \\
\leq &
u^2 \Var \left( \vk{w}_F^{\top} \vk{X}_{u,F}(t,s) - \vk{w}_F^{\top} \vk{X}_{u,F}(t_1 , s_1) \right) \\
\leq&
\frac{u^2}{2} \left[ \Var \left( \vk{w}_F^{\top} \vk{X}_F(u^{-2/\alpha} t) - \vk{w}_F^{\top} \vk{X}_F(u^{-2/\alpha} t_1) \right)
+
\Var \left( \vk{w}_F^{\top} \vk{X}_F(u^{-2/\alpha} s) - \vk{w}_F^{\top} \vk{X}_F(u^{-2/\alpha} s_1) \right)
\right] \\
=&
u^2 \vk{w}_F^{\top} \left( \Sigma_{FF} - \RR_{FF}(u^{-2/\alpha} (t-t_1)) \right) \vk{w}_F
+
u^2 \vk{w}_F^{\top} \left( \Sigma_{FF} - \RR_{FF}(u^{-2/\alpha} (s-s_1)) \right) \vk{w}_F \\
\leq &
c_5 \left( \left| t-t_1 \right|^{\alpha} + \left| s-s_1 \right|^{\alpha} \right).
\end{align*}
Now  Piterbarg inequality (c.f. \cite[Theorem~8.1]{Pit96}) and
\eqref{e:dsupF}, \eqref{e:duts} imply that
\begin{align*}
  {J}_u(\x)
  \leq &
  c_6 \left( \frac{\vk{w}_F^{\top} \vk{x}_F  - 2 c_4 \delta
	\lambda^{\alpha}}{\sigma_F} \right)^{4/\alpha} \exp \left( -
  \frac{\left(\vk{w}_F^{\top} \vk{x}_F - 2 c_4 \delta \lambda^{\alpha}\right)^2}{8
	\sigma_F^2} \right) \\
  \leq& {c_7 \exp \left( -
  \frac{\left(\vk{w}_F^{\top} \vk{x}_F - 2 c_4 \delta \lambda^{\alpha}\right)^2}{
  16 c_4 \delta \lambda^\alpha} \right)
  }
\end{align*}

for $\vk{w}_F^{\top} \vk{x}_F > 2 c_4 \delta \lambda^{\alpha}$.  It follows that 
\begin{equation}
\label{e:OmegaFint}
\begin{aligned}
& \int_{\Omega_F} e^{\left( \vk{w} + O({\lambda^{\alpha}}u^{-2})
	\right)^{\top} \vk{x}}
{J}_u(\x) d\vk{x} \\
\leq& c_8 \left( \delta \lambda^\alpha \right)^{|F|} e^{2 c_4 \delta \lambda^\alpha}
+ c_8 \int_{2 c_4 \delta \lambda^{\alpha}}^{\infty} y^{|F|-1}  \exp \left( 2y - \frac{\left( y - 2 c_4 \delta \lambda^{\alpha} \right)^2}{16 c_4 \delta \lambda^\alpha} \right) d y \\
\leq & c_9 \exp \left( c_{10} \delta \lambda^\alpha \right),
\end{aligned}
\end{equation}
{where $|F|$} is the cardinality of the set {$F$.} Together with \eqref{e:dsconditioned} and \eqref{e:Fempty} we obtain further
\begin{align*}
P_{\vk b}(\lambda, \Lambda, u)  \leq & c_{11} u^{-d} \varphi_{\Sigma}(u \vk{b}) \exp \left( -
\frac{\lambda^{\alpha}}{{8}} \xi_{\w}
+ c_{10}  \delta \lambda^{\alpha}  \right).
\end{align*}
Choosing $\delta>0$ small enough 
{so that $c_{10} \delta \leq \xi_{\w} / 16$},
we have
\[
  P_{\vk b}(\lambda, \Lambda, u) \leq c_{12} u^{-m} \varphi_{\Sigma}(u \vk{b}) \exp \left( - \frac{\lambda^{\alpha}}{{16}}\xi_{\w} \right) \sim c_{13} \pk{\X(0) > u \b} \exp \left( - \frac{\lambda^\alpha}{16} \xi_{\w} \right)\]
{as $u \to \infty$, which establishes} the proof. \QED

\begin{korr}
	\label{C:statn0}
	Under Assumption (B\ref{I:B2}), there are positive constants $C$ and $\varepsilon$ such that for every $\lambda > \Lambda > 0$ {with $\lambda + \Lambda < \varepsilon u^{2/\alpha}$} we have
	\begin{equation}
	\label{e:n0}
	\frac{P_{\vk b}(\lambda, \Lambda, u)  }
	{ \pk{\X(0) > u \b}}
	\leq
	C\Lambda^2 (\lambda - \Lambda)^{-2} e^{- \frac{(\lambda - \Lambda)^\alpha}{16} \xi_{\w}}.
	\end{equation}
\end{korr}

\begin{proof}[Proof of Corollary~\ref{C:statn0}]
{Hereafter  $\floor{x}$ stands for  the integer part of $x\inr$.} Let $n_0$ be the constant specified in Lemma~\ref{L:double}.	By Lemma~\ref{L:double}, it suffices to consider the case $\lambda < n_0 \Lambda$. Let $k_0 = \floor{\frac{n_0 \Lambda}{\lambda - \Lambda}} + 1$ and $\Lambda_0 = \Lambda / k_0$. For $0 \leq k$, $l \leq k_0 - 1$ we define
	\[
	A_{kl} = \left\{ \exists (t,s) \in u^{-2/\alpha} \left( [k \Lambda_0,\, (k+1) \Lambda_0] \times [\lambda + l \Lambda_0,\, \lambda + (l+1) \Lambda_0] \right):\ \X(t) > u \b,\, \X(s) > u \b \right\}
	\]
	and thus in this notation
	\[
	P_{\vk b}(\lambda, \Lambda, u)  \leq
	\sum_{k,l = 0}^{k_0 - 1} \pk{A_{kl}}.
	\]
	Since $\lambda + (l-k-1) \Lambda_0 \geq \lambda - \Lambda \geq n_0 \Lambda_0$  Lemma~\ref{L:double} implies
	\begin{align*}
	\pk{A_{kl}}
	= P_{\vk b}(\lambda + (l - k)\Lambda_0, \Lambda_0, u)
	\leq
	C \pk{\X(0) > u \b} e^{-\frac{(\lambda - \Lambda)^{\alpha}}{16} \xi_{\w}}.
	\end{align*}
	The claim  follows from the fact that $k_0^2 \leq 4 n_0^2 \Lambda^2 (\lambda-\Lambda)^{-2}$.
\end{proof}

\begin{proof}[Proof of Theorem~\ref{T:stat}]
	Let in the following for $\Lambda>0$, $u>0$
	\[ \Delta_k = [k \Lambda u^{-2/\alpha} , (k + 1) \Lambda u^{-2/\alpha} ], \quad 0 \leq k
	\leq N_T = \floor{\frac{T}{\Lambda u^{-2/\alpha}}}. \]
 Since $\X$  is stationary, we have
\begin{eqnarray}
\lefteqn{   (N_T + 1) \pk{\exists t \in \Delta_0:\ \X(t) > u \b}}\nonumber\\
  &\ge&
  \pk{\exists t \in [0, T]: \ \X(t) > u \b} \nonumber\\
  &\ge&
  N_T \pk{\exists t \in \Delta_0:\ \X(t) > u \b}
  -
  2 \sum_{k = 1}^{N_T} (N_T - k) \omega_{\b}(k \Lambda, \Lambda, u), \label{e:Psum}
\end{eqnarray}
where $P_{\vk b}(k \Lambda, \Lambda, u)$ is defined by \eqref{e:omegab}.
By Lemma~\ref{Lem1} and the stationarity of $\X$, as $u\to \IF$
	\[
	\pk{\exists t \in \Delta_k:\ \X(t) > u \b}
	\sim
	\H_{\Y, \d_{\w}}([0, \Lambda]) \pk{\X(0) > u \b},
	\]
 with $\vk Y$ as defined {at} the beginning of this section and
	\bqn{
		\d_{\w}(t) 
    = { \SA{t}{V_{\w}} \vk 1},  \quad V_{\w}= \diag{\w} V \diag{\w}.
	}
{Note that in our notation $\H_{\Y, \d_{\w}}([0, \Lambda])= \PICS ([0,\Lambda])$, consequently}
  \begin{equation}
\label{e:statub}
\limsup_{u \to \infty} \frac{\pk{\exists t \in [0, T]: \ \X(t) > u
		\b}} { T u^{2/\alpha}  \pk{ \vk X(0)> u \vk b} }
\leq   \frac{1}{\Lambda} 	\PICS  ([0,\Lambda]).
\end{equation}

The stationarity of $\X$ implies that the function $\Lambda \mapsto \PICS([0, \Lambda])$ is sub-additive. Therefore, the limit $$ \HA_{\alpha, V_{\w}} = \lim_{\Lambda \to \infty} \Lambda^{-1}\PICS([0, \Lambda]) $$
exists and is finite. {The  sum} in \eqref{e:Psum} is bounded by
\begin{align*}
  A_1 + A_2 + A_3 =
  N_T {P_{\b}}(\Lambda, \Lambda, u) +
  N_T \sum_{k=2}^{N_{\varepsilon}} {P_{\b}}(k \Lambda, \Lambda, u) +
  N_T \sum_{k = N_{\varepsilon} + 1}^{N_T} {P_{\b}}(k \Lambda, \Lambda, u).
\end{align*}
%
In view of Lemma~\ref{Lem1}, Corollary~\ref{C:statn0} and the Piterbarg inequality stated in Lemma~\ref{GBorell}
the negligibility of the double-sum follows with the same arguments as in the $1$-dimensional case.
Here we spell out the details for readers' convenience.

We first estimate $A_3$. For $k \geq N_{\varepsilon} + 1$, the
	distance between $\Delta_0$ and $\Delta_k$ is at least
	$\varepsilon/2$. Note that the variance matrix of $\X(t) +
	\X(s)$ is
	$$\Sigma(t, s) = 2 \Sigma + \RR(t - s) + \RR(s - t).$$
In view of  Assumption (B\ref{I:B1})   $\left(\Sigma_{II}(t,s)\right)^{-1} - \left(\Sigma_{II}\right)^{-1}$ is
	strictly positive definite for $t \neq s$, which implies that
	\begin{equation}
	\label{e:tau}
	\tau=\inf_{(t,s) \in \Delta_0 \times \Delta_k} \inf_{\v_I \geq
		\b_I} \v_I^{\top} \left(\Sigma(t,s)_{II}\right)^{-1} \v_I > \inf_{\v_I \geq \b_I} \v_I^{\top}
	\left(\Sigma_{II}\right)^{-1} \v_I = \b_I^{\top} \SIIIM \b_I {>0}.
	\end{equation}
	By the Piterbarg inequality stated in \nelem{GBorell}
	\begin{align*}
{P_{\b}}(k \Lambda, \Lambda, u)
    \leq & \pk{\exists (t,s) \in \Delta_0 \times \Delta_k
		:\ \X_{{I}}(t) + \X_{{I}}(s) > 2u \b_{{I}}} \\
	\leq & {c_1} \Lambda^2 u^{-2/\alpha}
	u^{\frac{4}{\gamma} - 1} \exp \left( - \frac{u^2}{2} \tau \right).
	\end{align*}
	It follows that
	\begin{equation}
	\label{e:A3}
	A_3 = O \left( \exp \left( - \frac{u^2}{2} (\tau - \delta) \right) \right)
	\end{equation}
	for some $0 < \delta < \tau - \b_I^{\top} \SIIIM \b_I$.	For $2 \leq k \leq N_{\varepsilon}$, we have from Corollary~\ref{C:statn0}
	that
\bqn{	\label{e:A2}
	\limsup_{u \to \infty} \frac{A_2}{T {u^{2/\alpha}} \pk{\vk X(0)> u \vk b}}
	&\leq &  {c_2 \Lambda}
	\sum_{k = 1}^{N_{\varepsilon}} \left( k \Lambda \right)^{-2}  \exp \left( - \frac{k^{\alpha}\Lambda^{\alpha}}{16} \xi_{\w}   \right) \notag \\
&	\leq& {c_3} \Lambda^{-1} \exp \left( -\frac{\Lambda^{\alpha}}{16} \xi_{\w} \right).
}
	
	Now we consider $A_1$. Note that
  \begin{align*}
    {P_{\b}}(\Lambda, \Lambda, u)
    \leq
    {P_{\b}}(\Lambda + \sqrt{\Lambda}, \Lambda, u) +
    \pk{\exists t \in u^{-2/\alpha} [0, \sqrt{\Lambda}] :\ \X(t)
		> u \b}.
  \end{align*}
	Applying Corollary~\ref{C:statn0}, \nelem{Lem1} and the subadditivity of $\PICS([0, \Lambda])$ we obtain
	\begin{equation}
	\label{e:A1}
	\limsup_{u \to \infty} \frac{A_1}{{T u^{2/\alpha}} \pk{\vk X(0)> u \vk b}} \leq {c_5}  \left[ \exp 	\left( - \frac{\Lambda^{\alpha/2}}{16} \xi_{\w} \right) +
	\PICS([0, 1]) \Lambda^{-1/2}
	\right].
	\end{equation}

	Putting all the bounds \eqref{e:Psum}--\eqref{e:A1} together,  for any $\Lambda_1$ and $\Lambda_2 > 0$ we
	obtain that
\BQN
	\label{e:supinf}
 \frac{ \PICS([0, \Lambda_1])} {\Lambda_1} 
	&\geq & \limsup_{u \to \infty} \frac{\pk{\exists t \in [0, T]: \ \X(t) > u
			\b}} {T {u^{2/\alpha}}  \pk{ \vk X(0)> u \vk b }}
	\notag \\
	&\geq & \liminf_{u \to \infty} \frac{\pk{\exists t \in [0, T]: \ \X(t) > u
			\b}} {T {u^{2/\alpha}} \pk{ \vk X(0)> u \vk b }} \notag \\
	&\geq &
	\frac{\PICS([0, \Lambda_2])}{\Lambda_2}
	- {c_6} \Lambda_2^{-1} \exp \left( -\frac{\Lambda_2^{\alpha}}{16} \xi_{\w} \right) \notag\\
	&&- {c_7}  \exp \left( - \frac{ \Lambda_2^{\alpha/2}}{16} \xi_{\w} \right)  - {c_8} 
 \PICS([0, 1]) 	\Lambda_2^{-1/2}.
\EQN
Consequently, the constant $ \HA_{\alpha, V_{\w}}$ is positive.
	This and \eqref{e:supinf} establish the proof when $T$ does not depend on $u$.
The case $T$ dependent on $u$ follows with {analogous} calculations.
\end{proof}

\subsection{Proof of Theorem \ref{T:Vskew}}
\label{s:Vskew}

\begin{lem}
	\label{L:integral}
	For every $\v \in \R^d$ such that  $a = \id^{\top} \v \geq 0$ and $\Lambda \geq 0$ we have
	\[ \int_{\R^d} e^{\id^{\top} \x} \id_{\{ \exists t \in [0, \Lambda]:\ \x < - t \v\}} \kk{d\x} =
	\begin{cases}
	1 + \Lambda \sum_{i=1}^d v_i^-, & \text{if } a = 0, \\
	1 + \frac{1 - e^{-a \Lambda}}{a} \sum_{i=1}^d v_i^{-} & \text{if } a > 0,
	\end{cases}
	\]
{where $x^{-}= \max(0,-x)$, $x\inr$.}
\end{lem}

\begin{proof}[Proof of Lemma~\ref{L:integral}]
	Set $b = \sum_{i=1}^d v_i^-$ and
$$F = \{ 1 \leq i \leq d:\ v_i > 0 \},\quad  \bF = \left\{ 1,\ldots, d \right\} \setminus F.$$
Define $\Omega_t = \{\x < - t \v \}$ and let $\Omega = \bigcup_{0 \leq t \leq \Lambda} \Omega_t$. For a Borel set $A \subset \R^d$, let $\I(A) = \int_A e^{\id^{\top} \x} d \x$. Then for every $t\geq 0$ we have $\I(\Omega_t) = e^{-a t}$. Note that for $t < s$, $\Omega_s \cap \Omega_t = \left\{ \x_F < - s \v_F,\ \x_{\bF} < - t \v_{\bF} \right\}$. Since $\I(\Omega_s \cap \Omega_t) = e^{-(a+b)s + b t}$, 
	$$\I(\Omega_s \setminus \Omega_t) = e^{-a t - (a+b)(s-t)} \left( e^{b t} - 1 \right).$$
	 On the other hand, $\bigcup_{t < u < s} \Omega_u \subset \left\{ \x_F < - t \v_F,\ \x_{\bF} < -s \v_{\bF} \right\}$ and hence
	\[ \I\left(\bigcup_{t < u < s} \Omega_u \setminus \Omega_t\right) \leq e^{-a t} \left( e^{b(s-t)} - 1 \right). \]
	Let $t_k = k \Lambda / n$. Since
	\[ \bigcup_{0 \leq k \leq n-1} \left( \Omega_{t_{k+1}} \setminus \Omega_{t_k} \right) \subset \Omega \setminus \Omega_0 \subset \bigcup_{0\leq k \leq n-1} \left[ \left( \bigcup_{t_{k+1} \leq u \leq t_k} \Omega_u \right) \setminus \Omega_{t_k}\right]
	 \]
	we have
	\[
	\sum_{k=0}^{n-1} e^{-a t_k - \frac{\Lambda (a+b)}{n}} \left( e^{\frac{\Lambda b}{n}} - 1 \right) \leq \I(\Omega) - 1 \leq \sum_{k=0}^{n-1} e^{-a t_k} \left( e^{\frac{\Lambda b}{n}} - 1 \right).
	\]
	Letting $n \to \infty$ we complete the proof.
\end{proof}

\begin{proof}[Proof of Theorem \ref{T:Vskew}]
	We follow the same idea as \kk{in} the proof of Theorem~\ref{T:stat} and only spell out the necessary changes.
	
	For $\Lambda > 0$, from Lemma~\ref{Lem1} and the fact that $V^{\top} = - V$ we obtain
	\[
	\lim_{u \to \infty} \frac{\pk{\exists t \in [0, \Lambda u^{-2}]:\ \X(t) > u \b}}{\pk{\X(0) > u \b}} = \H_{0,\d_{\w}}([0, \Lambda]),
	\]
	where $\d_{\w}(t) = t \diag{\w} V \w$ and the constant $\H_{0,\d_{\w}}([0, \Lambda])$ is given
{by (\ref{def.H})}.
By Lemma~\ref{L:integral}, we have that
	\begin{equation}
	\label{e:pickVskew}
	\lim_{\Lambda \to \infty} \frac{\H_{0,\d_{\w}}([0, \Lambda])}{\Lambda} = \frac{\sum_{i \in I} w_i \left| (V \w)_i \right|}{2}>0.
	\end{equation}
	 Let below
	 $$M = \left\{ i \in I:\ \left( V \w \right)_i > 0 \right\}, \quad \overline{M} = I \setminus M.$$
	  {By the assumption $\left( V \w \right)_I \neq \vk{0}_I$}, both $M$ and $\overline{M}$ are non-empty.
	{Using that (\ref{e:Psum}) also holds under the conditions of Theorem \ref{T:Vskew},
	now we analyze  the sum  in  (\ref{e:Psum}).}
    Without loss of generality, we assume that $I = \{1,\ldots, d\}$.
	Define
	\[ \Delta_k = {[0,\  \Lambda]} \times [k  \Lambda,\ (k+1)  \Lambda], \quad 0 < k \leq N_T = \frac{T u^2}{\Lambda}.
	\]
We first estimate the sum in~\eqref{e:Psum} for large $k$.	 {By} \eqref{e:B2'}, for any $\delta > 0$, there exists $\varepsilon > 0$ such that $\normF{\Sigma - R(t) - t V} \leq \delta \lambda$ for $|t| < \varepsilon$. Assume that $(k+1) \Lambda \leq \varepsilon u^2$. Let $\vk{\chi}_u(t)$ be the conditioned process $u (\X(u^{-2} t) - u \b) + \x$ given $\X(0) = u \b - u^{-2} \x$. Then
	\begin{align*}
	& \pk{\exists (t,s) \in u^{-2} \Delta_k:\ \X(t) > u \b,\ \X(s) > u \b} \\
	\leq & \pk{\X(0) > u \b} \int_{\R^d} e^{\w^{\top} \x} \pk{\exists (t,s) \in \Delta_k:\ \vk{\chi}_u(t) > \x,\ \vk{\chi}_u(s) > \x} d \x.
	\end{align*}
	Setting  $\d_u(t) = - \E{\vk{\chi}_u(t)}$ and $\d(t) = t V \w$ we have
	\[ \left| \d_u(t) - \d(t) - O(\lambda u^{-2}) \x \right| \leq c_1 \delta \lambda \]
	and
	\[ \normF{\E{(\vk{\chi}_u(t) + \d_u(t)) (\vk{\chi}_u(s) + \d_u(s))^{\top}}} \leq c_1 \delta \lambda \]
	for some $c_1 > 0$. Consequently, for all 	 $(t,s) \in \Delta_k$
	\[ \d_{u,M}(s) \geq \lambda (V \w)_M - c_1 \delta \lambda - O(\lambda u^{-2}) \x_M,
	$$
	$$
	\d_{u,\overline{M}}(t) \geq \Lambda (V \w)_{\overline{M}} - c_1 \delta \lambda - O(\lambda u^{-2}) \x_{\overline{M}}.
	\]
	By the change of variables $\y_M = \x_M - \lambda (V \w)_M + c_1 \delta \lambda$ and $\y_{\overline{M}} = \x_{\overline{M}} - \Lambda (V \w)_{\overline{M}} + c_1 \delta \lambda$ and the assumption  $\w^{\top} V \w = 0$, we have that
	\begin{align*}
	& \int_{\R^d} e^{\w^{\top} \x} \pk{\exists (t,s) \in \Delta_k:\ \vk{\chi}_u(t) > \x,\ \vk{\chi}_u(s) > \x} d \x \\
	\leq & \int_{\R^d} d\x \, e^{\w^{\top} \x} \, \mathbb{P}\left\{\exists (t,s) \in \Delta_k:\ \vk{\chi}_{u, F}(s) > \x_M - \lambda (V \w)_M + c_1 \delta \lambda + O(\lambda u^{-2}) \x_M, \right. \\
	& \qquad \qquad \qquad \qquad    \left.  \vk{\chi}_{u,\overline{M}}(t) > \x_{\overline{M}} - \Lambda (V \w)_{\overline{M}} + c_1 \delta \lambda + O(\lambda u^{-2}) \x_{\overline{M}}\right\} \\
	\leq & e^{- \w_M^{\top} (V \w)_M (\lambda - \Lambda) + c_2 \delta \lambda} \int_{\R^d} e^{(\w + O(\lambda u^{-2}))^{\top} \y} g_u(\y) d \y,
	\end{align*}
	where
	\[ g_u(\y) = \pk{\exists (t,s) \in [0, \Lambda] \times [0, \Lambda]:\ \Z_{u,M}(s) > \y_F, \ \Z_{u,\overline{M}}(t) > \y_{\overline{M}}} \]
	and
	\[ \Z_{u,M}(t) = \vk{\chi}_{u,M}(t + k \Lambda) - \d_{u,M}(t + k \Lambda), \quad \Z_{u,\overline{M}}(t) = \vk{\chi}_{u,\overline{M}}(t) - \d_{u,\overline{M}}(t). \]
	For $F \subset \{1,\ldots,d \}$ with $m = |F|$ let
	$$\Omega_F = \{\y \in \R^d:\ y_i > 0,\ y_j < 0,\ i \in F,\ j \not \in F\}$$ and set
	$$\eta(t_1,\ldots, t_m) = \sum_{i \in F} \Z_{u,i}(t_i).$$
	Note that
	\[ \sup_{t_i \in [0, \Lambda], {i\le m}} \E{\eta(t_1, \ldots, t_m)^2} \leq c_3 \delta \lambda. \]
	Applying  Borell-TIS inequality we obtain
	\[ \int_{\R^d} e^{(\w + O(\lambda u^{-2}))^{\top} \y} g_u(\y) d \y \leq e^{c_4 \delta \lambda}. \]
	Therefore, there exists $n_0 > 0$ such that for all $k \geq n_0$ and $(k+1) \Lambda \leq \varepsilon u^2$
	\[ \pk{\exists (t,s) \in u^{-2} \Delta_k:\ \X(t) > u \b,\ \X(s) > u \b} \leq \pk{\X(0) > u \b} e^{- \frac{\w_M^{\top} (V \w)_M}{2} \lambda}. \]
	
It remains to estimate the sum in~\eqref{e:Psum} for $1 \leq k \leq n_0$. We have {from Lemma~\ref{Lem1}} that
	\begin{align*}
	& \limsup_{u \to \infty} \frac{\pk{\exists (t,s) \in \Delta_k:\ \X(t) > u \b,\ \X(s) > u \b}}{\pk{\X(0) > u \b}} \\
	\leq &
	\int_{\R^d} e^{\w^{\top} \x} \pk{\exists (t,s) \in \Delta_k:\ - \d(t) > \x,\ - \d(s) > \x }  d \x\\
	\leq &
	\int_{\R^d} e^{\w^{\top} \x} \id_{\{ \x_M < - \lambda (V \w)_M,\ \x_{\overline{M}} < - \Lambda (V \w)_{\overline{M}}\}} d \x \\
	\leq &
	c_5 e^{- \w_M^{\top} (V \w)_M (\lambda - \Lambda)},
	\end{align*}
	where $\d(t) = t V \w$.
	
{From the above we conclude the negligibility of the  sum in (\ref{e:Psum}).
The rest of the proof follows by
the same arguments as given in the proof of Theorem~\ref{T:stat}.}
\end{proof}

\subsection{Proof of \netheo{T:nonstationary}}

We present first several supporting lemmas and then continue with the proof of
\netheo{T:nonstationary}.

For the next two lemmas we impose the assumptions of \netheo{T:nonstationary}.
Set below $\delta_u = u^{-2/\beta} \log^{2/\beta} u$ and  recall that in view of  \eqref{e:bAXi}
$$ \tau_{\vk w}=  \vk{w}^{\top} \Xi A^{\top} \vk{w} =
\vk{w}^{\top}_I (\Xi A^{\top})_{II} \vk{w}_I  > 0.$$

\begin{lem}
	\label{L:largek}
  {There exist positive} constants $C$, $u_0$ and $\Lambda_0$ such that for $\Lambda \geq \Lambda_0$ and $u \geq u_0$
	\begin{equation}
	\label{e:largek}
	\pk{\exists t \in [\Lambda u^{-\frac{2}{\beta}},\ \delta_u]:\ \X(t) > u \b} \leq C \exp \left( -  \tau_{\w} \Lambda^{\beta} \right) \pk{\X(0) > u \b}.
	\end{equation}
\end{lem}

\prooflem{L:largek}  For simplicity we shall assume that $|I| = d$ ({thus $\tilde{\b}= \vk b$ below)}. 
Letting  $\nu = \min( \alpha , \beta)$ and $\theta_u = \Lambda u^{-2/\nu}$.
For $1 \leq k \leq N_u=
\ceil{\delta_u/\theta_u}$ we define $\vk{X}_{u,k}(t) = \vk{X}(k \theta_u + t
u^{-2/\nu})$. Then
\[
R_{u,k}(t,s) = \E{\vk{X}_{u,k}(t) \vk{X}_{u,k}(s)^{\top}} = R(k \theta_u + t
u^{-2/\nu}, k \theta_u + s u^{-2/\nu}),
\]
where  $R(t,s)$ is the cmf of $\X$. Setting  next
$\Sigma_{u,k} = \E{\vk{X}_{u,k}(0) \vk{X}_{u,k}(0)^{\top}} = \Sigma(k
\theta_u)$ we have (see also the proof of \nelem{Lem1})
\begin{align*}
& \pk{\exists t \in [k \theta_u, (k+1) \theta_u]:\ \X(t) > u \vk{b}} \\
\leq &
u^{-d} \varphi_{\Sigma_{u,k}}(u \vk{b}) \int_{\R^d} e^{\vk{b}^{\top} \Sigma_{u,k}^{-1} \vk{x}}
\pk{ \exists t \in [0, \Lambda]:\
	\vk{\chi}_{u,k}(t) > \x}
d \vk{x},
\end{align*}
where $\vk{\chi}_{u,k}(t)$ is the conditional process $u \left( \X_{u,k}(t) - u
\vk{b} \right) + \vk{x}$ given $\vk{X}_{u,k}(0) = u \vk{b} - u^{-1}\vk{x}$.
Note that
\[
\vk{d}_{u,k}(t)=
- \E{ \vk{\chi}_{u,k}(t)}
=
u^2 \left[ \Sigma_{u,k}  - R_{u,k}(t,0) \right] \Sigma_{u,k}^{-1} \vk{b}
+
\left[ \Sigma_{u,k} - R_{u,k}(t,0) \right] \Sigma_{u,k}^{-1} \vk{x}
\]
and
\begin{align*}
K_{u,k}(t,s) =& \E{ \left[ \vk{\chi}_{u,k}(t) - \vk{d}_{u,k}(t) \right]
	\left[ \vk{\chi}_{u,k}(s) - \vk{d}_{u,k}(t) \right]^{\top}} \\
=&
u^2 \left[ R_{u,k}(t,s) - R_{u,k}(t,0) \Sigma_{u,k}^{-1} R_{u,k}(0,s) \right].
\end{align*}
By Assumptions (D\ref{I:D2}) and (D\ref{I:D3}) {with $R_{\alpha,V}$ defined in \eqref{e:KV}}
\begin{equation}
\label{e:limK}
 K_{u,k}(t,s) \to
\begin{cases}
R_{\alpha,V}(t,s), &
\text{ if\ } \beta \geq \alpha \\
0, & \text{ if\ } \beta < \alpha
\end{cases}
\end{equation}
as ${u \to \infty}$, where the convergence is uniform in $(t,s) \in [0, \Lambda] \times [0, \Lambda]$ and
$1 \leq k \leq N_u$.

By Assumptions (D\ref{I:D2}) and (D\ref{I:D3}) again,  for every $\varepsilon > 0$, there is $u_0$ such that for $u \geq u_0$
\begin{equation}
\label{e:subdukt}
\left| \vk{d}_{u,k} (t) - \vk{d}(t) - \left[ \Sigma_{u,k} - R_{u,k}(t,0) \right]
\Sigma_{u,k}^{-1} \vk{x} \right| \leq \varepsilon {u_*}  k^{\beta} \Lambda^{\beta},
\end{equation}
with {$u_* = \min( 1, u^{2-\frac{2\beta}{\alpha}} )$ and }
\[
\vk{d}(t) =
\begin{cases}
|t|^{\alpha} V \vk{w} , & \text{if\ } \beta > \alpha, \\
\left[ \left( k \Lambda + t \right)^{\alpha} - \left( k \Lambda
\right)^{\alpha} \right] \Xi A^{\top} \vk{w} + |t|^{\alpha} V
\vk{w}, &
\text{if\ } \beta = \alpha, \\
\left[ \left( k \Lambda + t \right)^{\beta} - \left( k \Lambda
\right)^{\beta} \right] \Xi A^{\top} \vk{w}, & \text{if\ } \beta < \alpha.
\end{cases}
\]
In the above derivation we used \eqref{aero} with $h=\beta$ for the case $\beta > \alpha$.
Consequently,
\[ \lim_{u \to
	\infty} {\max_{   1\le k \le N_u}} \ABs{ u^{2 - 2\beta/\alpha} \left( \left( k \Lambda + t \right)^{\beta}
- (k \Lambda)^{\beta}\right) }= 0.
 \]

As in the proof of Lemma~\ref{L:double}, we define  for $F \subset
\{1,\ldots, d\}$
$$\Omega_F = \{\vk{x} \in
\R^d:\ x_i > 0,\, x_j < 0,\, i \in F \text{ and\ } j \not\in F\}.$$
Applying \eqref{aero} we have that
\[
  \sup_{t \in [0, \Lambda]} \left| \vk{w}_F^{\top} \vk{d}_F(t) \right| \leq
\begin{cases}
c_1 \Lambda^{\alpha}, & \text{if\ } \beta > \alpha, \\
c_1 \left( 1 + k^{\beta-1} \right) \Lambda^{\beta}, & \text{if\ } \beta \leq
\alpha.
\end{cases}
\]
It follows from \eqref{e:subdukt} that 
for every $1 \leq k \leq  N_u $ and all $u$ large enough
\begin{equation}
\label{e:subsupd}
\sup_{t \in [0, \Lambda]} \left| \vk{w}_F^{\top} \vk{d}_{u,k,F}(t)\right|
\leq
\frac{1}{2} \vk{w}_F^{\top} \vk{x}_F + c(k,\Lambda),
\end{equation}
with
\[
c(k,\Lambda) = \begin{cases}
c_1 \Lambda^{\alpha} +  \varepsilon {u_*} k^\beta \Lambda^\beta, 
 & \text{if\ } \beta >
\alpha, \\
c_1 \left( 1 + k^{\beta-1} \right) \Lambda^\beta + \varepsilon k^{\beta} \Lambda^{\beta}, & \text{if\ } \beta \leq \alpha.
\end{cases}
\]

Setting  $\eta_{u,k}(t) = \vk{w}_F^{\top} \left( \vk{\chi}_{u,k}(t) +
\vk{d}_{u,k}(t) \right)$  for every
$\vk{x} \in \Omega_F$  we have
\begin{align*}
\pk{ \exists t \in [0, \Lambda]:\
	\vk{\chi}_{u,k}(t) > \x}
\leq
\pk{\sup_{t \in [0, \Lambda]} \eta_{u,k}(t) > \frac{1}{2} \vk{w}_F^{\top} \vk{x}_F -  c(k,\Lambda)}.
\end{align*}
By \eqref{e:limK}, the variance of $\eta_{u,k}(t)$,
$0 \leq t \leq \Lambda$ is bounded uniformly (with respect to $k\le N_u$) by $\sigma^2 = c_2
\Lambda^{\alpha}$ for $\beta \geq \alpha$ and $\sigma^2 = c_2$ for $\beta <
\alpha$. Consequently, Piterbarg inequality  implies
\[
  \pk{ \exists t \in [0, \Lambda]:\
    \vk{\chi}_{u,k}(t) > \x}
  \leq
  c_3 \left( \frac{\vk{w}_F^{\top} \vk{x}_F - 2 c(k,\Lambda)}{\sigma} \right)^{2/\gamma}
  \exp \left( - \frac{\left(\vk{w}_F^{\top} \vk{x}_F - 2 c(k,\Lambda)
      \right)^2}{8  \sigma^2} \right).
\]
Similarly  to the derivation of  \eqref{e:OmegaFint} in the proof of Lemma~\ref{L:double} it follows that 
\[
\int_{\R^d} e^{\vk{b}^{\top} \Sigma_{u,k}^{-1} \vk{x}} \pk{\exists t \in
	[0, \Lambda]:\ \vk{\chi}_{u,k}(t) > \vk{x}} d \vk{x}
\leq
c_4  e^{c_5 \left( \sigma^2 + c(k,\Lambda) \right)}
{ \leq c_6 e^{c_5 c(k,\, \Lambda)}}.
\]

Assumption (D\ref{I:D2}) implies further 
\begin{equation}
  \label{e:subSigmauk}
  u^2 \b^\top \left( \Sigma_{u,k}^{-1} - \Sigma^{-1} \right) \b
  = u^2 \b^\top \Sigma_{u,k}^{-1} \left( \Sigma - \Sigma_{u,k} \right) \Sigma^{-1} \b
  =
  2 \tau_{\w}  u_{*} k^{\beta} \Lambda^{\beta}  + o \left( u_{*} k^{\beta} \Lambda^{\beta} \right).
\end{equation}
It follows that
\begin{equation}
\label{e:k0}
\frac{\pk{\exists t \in [k \theta_u, (k+1) \theta_u]:\ \vk{X}(t) > u \vk{b}}}{\pk{\X(0) > u \b}}
\leq
{
  c_7  \exp \left( - \frac{3}{2} \tau_{\w} u_* k^\beta \Lambda^\beta + c(k, \Lambda) \right)
}
\end{equation}
for $u$ large enough. {If $\beta > \alpha$, then the left-hand side of~\eqref{e:largek} is at most
  \begin{align*}
    \sum_{k = \lfloor u^{\frac{2}{\alpha} - \frac{2}{\beta}} \rfloor}^{N_u} \pk{\exists t \in [k \theta_u, (k+1) \theta_u]:\ \X(t) > u \b}
    \leq
           c_8 \pk{\X(0) > u \b} \exp \left(- \frac{3}{2} \tau_{\w} \Lambda^\beta + c_1 \Lambda^\alpha \right),
  \end{align*}
  {hence the thesis of} the lemma follows by taking $\Lambda \geq \Lambda_0$ with $\Lambda_0^{\beta - \alpha} \tau_{\w} > 2 c_1$.
}

{Now assume that $\beta \leq \alpha$. Choose $k_0$ so that $c_1(k_0^{-\beta} + k_0^{-1}) < \tau_{\w} / 2$. By \eqref{e:k0}, we have for $k > k_0$
  \begin{equation}
    \label{e:k0beta}
    \pk{\exists t \in [k \theta_u, (k+1) \theta_u]: \X(t) > u \b} \leq c_8 \pk{\X(0) > u \b} \exp \left( - \tau_{\w} k^\beta \Lambda^\beta \right).
  \end{equation}
}
It remains to consider the case $1 \leq k \leq k_0$. Set $\widetilde{\Lambda} = \Lambda /
k_0$ and  note that our choice of  $k_0$ is independent of $\Lambda$. By
\eqref{e:k0beta} we have that
\begin{align*}
 \pk{\exists t \in [k \theta_u, (k+1) \theta_u]:\ \vk{X}(t) > u \vk{b}}
\leq &
       \sum_{j=k_0 k}^{k_0(k+1)-1} \pk{\exists t \in  [j \widetilde{\Lambda} u^{-2/\nu}, (j+1 )\widetilde{\Lambda} u^{-2/\nu}]:\ \vk{X}(t) > u\vk{b}} \\
\leq&
\sum_{j=k_0k}^{k_0(k+1)-1} c_8 \pk{\X(0) > u \b} \exp \left( - \frac{1}{2} \tau_{\vk w} u_* j^{\beta} \widetilde{\Lambda}^{\beta}  \right) \\
\leq &
       c_9  \pk{\X(0) > u \b} \exp \left( - \frac{1}{2} \tau_{\vk w}u_*  k^{\beta} \Lambda^{\beta}  \right),
\end{align*}
which together with  \eqref{e:k0beta} completes the proof. \QED

\begin{korr}
	\label{C:HYdL}
{If $\alpha, V, \vk w$ and $W= \Xi A^\top$ are}  as in  Theorem~\ref{T:nonstationary}, then
$\mathcal{P}_{\alpha,V_{\w},W_{\w}}{\in(0,\infty)}$.
\end{korr}

\proofkorr{C:HYdL}
First, {we} note that with $\vk d_{\w}(t)= { \diag{\w} V (t) \w} $ and $\vk Y$ as
in  Theorem~\ref{T:nonstationary}, we have
 $\mathcal{P}_{\alpha, V_{\w}, W_{\vk w} } = \lim_{\Lambda \to \infty}  \H_{\Y, \vk{d}_{\w}+\vk{f}_{\w}}([0,\Lambda])$,
  where
 \bqn{ \label{fw}
\quad \quad  	V_{\w} = \diag{\w} V  \diag{\w},    \quad 	W_{\w} = \diag{\w} \Xi A^\top   \diag{\w},
 	\vk f_{\w}(t) = \abs{t}^\alpha W_{\w} \id .
 }

 Since {$\H_{\Y, \vk{d}_{\w}+\vk{f}_{\w}}([0,\Lambda])$} is increasing in $\Lambda$, it suffices to
 prove that {it} 
 is uniformly bounded. 

Fix $0 < \Lambda_0 < \Lambda$. 
By Lemma~\ref{L:largek} for two positive constants  $c$, $C$
\begin{align*}
\limsup_{u \to \infty} \frac{\pk{\exists t \in [\Lambda_0 u^{-2/\alpha}, \Lambda u ^{-2/\alpha}]:\ \vk{X}(t) > u \vk{b}}}{
	\pk{ \vk X(0)> u \vk b}} \leq C e^{- c \Lambda_0^{\alpha}}.
\end{align*}
It follows that
\begin{align*}
  \H_{\Y, \vk{d}_{\w}+\vk{f}_{\w}}([0,\Lambda])
=&
\lim_{u \to \infty} \frac{\pk{\exists t \in [0, \Lambda u ^{-2/\alpha}]:\ \vk{X}(t) > u \vk{b}}}{\pk{ \vk X(0)> u \vk b}}\\
\leq&
\lim_{u \to \infty} \frac{\pk{\exists t \in [0, \Lambda_0 u ^{-2/\alpha}]:\ \vk{X}(t) > u \vk{b}}}{\pk{ \vk X(0)> u \vk b}}
      + C e^{-c \Lambda_0^{\alpha}}
  \\
  \leq &
         \H_{\Y, \vk{d}_{\w}+\vk{f}_{\w}}([0,\Lambda_0]) + C e^{-c \Lambda_0^{\alpha}},
\end{align*}
which completes the proof.  \QED

Set in the following $\Delta(\tau, \lambda, \Lambda)=  [\tau, \tau+\Lambda] \times
[\lambda, \lambda + \Lambda]$
and
$$ P_{\vk b}( \tau, \lambda, \Lambda,u) =	\pk{ \exists (t,s) \in u^{-2/\alpha} \Delta(\tau, \lambda, \Lambda):\ \vk{X}(t) > u \vk{b}, \ \vk{X}(s) >
	u \vk{b}}
.$$
\begin{lem}
	\label{L:doublesub}
 If $\beta > \alpha$ and $\xi_{\w} = \w^{\top} V \w > 0$, then  for every
	$0 <\tau+\Lambda < \lambda \leq N_u$ with $u$ large enough
	\begin{align*}
P_{\vk b}( \tau, \lambda, \Lambda,u)	\leq& C_1 \Lambda^2 \exp \left( - C_2 \left( \lambda - \tau - \Lambda \right)^{\alpha}
	\right) u^{-d} \varphi_{\Sigma_{u,\tau,\lambda}}(u \vk{b}),
	\end{align*}
	where  $C_1,C_2$ are two positive constants and
	\begin{equation}
	\label{e:Sigmautl}
	\Sigma_{u,\tau,\lambda} = \frac{1}{4} \E{( \vk{X}(u^{-\alpha/2}\tau) +
		\vk{X}(u^{-\alpha/2}\lambda) ) ( \vk{X}(u^{-\alpha/2}\tau) +
		\vk{X}(u^{-\alpha/2}\lambda) )^{\top}}.
	\end{equation}
\end{lem}

\prooflem{L:doublesub}
The proof is similar to that of Lemma~\ref{L:double} 
and we only sketch the main ideas.
We shall assume for simplicity that $|I| = d$.
Set $\vk{X}_u(t,s) = \frac{1}{2} \left( \vk{X}(u^{-\alpha/2}t) +
\vk{X}(u^{-\alpha/2}  s) \right)$ and define
$$  	 P(\tau, \lambda,\Lambda) =\pk{ \exists (t,s) \in u^{-\alpha/2} \Delta(\tau, \lambda, \Lambda):\ \vk{X}(t) + \vk{X}(s) >
	2 u \vk{b}}.$$
For any $u>0$
\[\frac{ u^d P(\tau, \lambda,\Lambda)} {\varphi_{\Sigma_{u,\tau,\lambda}}(u \vk{b})}  \leq
 \int_{\R^d}
e^{\vk{b}^{\top} \Sigma_{u,\tau,\lambda}^{-1} \vk{x}} \pk{ \exists (t,s) \in \Delta(\tau, \lambda, \Lambda):\
	\vk{\chi}_{u,\tau,\lambda}(t,s) >  \vk{x}} d \vk{x},
\]
where $\vk{\chi}_{u,\tau,\lambda}(t,s)$ is the conditioned process  $u \left(
\vk{X}_u(t,s) - u \vk{b} \right) + \vk{x}$ given
$\vk{X}_u(\tau,\lambda) = u \vk{b} - u^{-1} \vk{x}$. Define $\d_{u,\tau,\lambda}(t) = - \E{\vk{\chi}_{u,\tau,\lambda}(t)}$ and let $R_u(t,s;t_1,s_1)$ be the cmf of $\vk{\chi}_{u,\tau,\lambda}$. Set further
\[
d(t,s) =  \frac{1}{4} \left[ (t-\tau)^{\alpha} + (s-\tau)^{\alpha} +
(\lambda-t)^{\alpha} + (s - \lambda)^{\alpha} - 2 (\lambda - \tau)^{\alpha}
\right]
\]
and
\begin{align*}
r(t,s;t_1,s_1) =& (t-\tau)^{\alpha} + \left( s - \tau \right)^{\alpha}
+ (t_1 - \tau)^{\alpha}\\
&  + (s_1-\tau)^{\alpha} + (\lambda - t)^{\alpha} + (\lambda-t_1)^{\alpha} + (s-\lambda)^{\alpha} + (s_1 - \lambda)^{\alpha} \\
&  - |t-t_1|^{\alpha} - |s-s_1|^{\alpha} - (s_1-t)^{\alpha} - (s-t_1)^{\alpha} - 2 (\lambda - \tau)^{\alpha}.
\end{align*}
By Assumptions~(D\ref{I:D2}) and (D\ref{I:D3}), we have that for
every $\varepsilon > 0$ and $(t,s) \in \Delta(\tau,\lambda,\Lambda)$
\begin{equation}
\label{e:dudts}
\left| \vk{d}_{u,\tau,\lambda}(t,s) - d(t,s) V \vk{w} -  \left[ \Sigma_{u,\tau,\lambda} - R_u(t,s;\tau,\lambda)  \right]
\Sigma_{u,\tau,\lambda}^{-1} \vk{x} \right| \leq \varepsilon \left( \lambda - \tau \right)^{\alpha}
\end{equation}
and
\begin{equation}
\label{e:KuRts}
\left\Vert R_u(t,s;t_1,s_1) - \frac{1}{4} r(t,s;t_1,s_1) V \right\Vert_{\mathrm{F}}
\leq
\varepsilon \left( \lambda - \tau \right)^{\alpha},
\end{equation}
provided $u$ is large enough. Now, by the same {argument as in the proof of} Lemma~\ref{L:double}, there exist positive
constants $c_1$ and $n_0$ such that for every $\lambda - \tau \geq n_0 \Lambda >
0$ with $\lambda \leq N_u$ and $u$ large enough
\begin{equation}
\label{e:lambdalarge}
P(\tau, \lambda,\Lambda)
\leq C_1 u^{-d} \varphi_{\Sigma_{u,\tau,\lambda}}(u \vk{b}) \exp \left( -\xi_{\vk w}^2\frac{(\lambda-\tau)^{\alpha}}{16}  \right).
\end{equation}

It remains to consider the case $\lambda - \tau \leq n_0 \Lambda$. Set
$\vk{X}_{u,\tau}(t) = \vk{X}(u^{-\alpha/2}(\tau + t))$ with cmf $R_{u,\tau} (t,s) = R(u^{-\alpha/2}(\tau+t), u^{-\alpha/2}(\tau+s))$ and define $\Sigma_{u,\tau}=R_{u,\tau}(0,0)$. 
Note that
\bqn{
\lim_{u\to\infty} u^2 \left[ \Sigma_{u,\tau} - R_{u,\tau}(t,0) \right]
= |t|^{\alpha} V
}
and
\[ \lim_{u\to\infty} u^2 \left[ R_{u,\tau}(t,s) - R_{u,\tau}(t,0)
\Sigma_{u,\tau}^{-1} R_{u,\tau}(0,s) \right] 
= R_{\alpha,V}(t,s)
\]
uniformly in $\lambda - \tau \leq n_0 \Lambda$. Recall that we defined  $\vk Y$ as  a  centered $\R^d$-valued Gaussian  process with cmf $\diag{\w} R_{\alpha,V} \diag{\w}$ and
$$\d_{\w}(t) = |t|^{\alpha} \diag{\w} V \w= V_{\w} \id, V_{\vk w}= \diag{\w} V \diag{\w}.$$
Analogous to 	Lemma~\ref{Lem1}
\BQNY
P_{\vk b}( \tau, \lambda, \Lambda,u)	&\sim & \H(\lambda-\tau, \Lambda) u^{-d} \varphi_{\Sigma_{u,\tau,\lambda}}(u \tilb)
\EQNY	as $u \to \infty$, where (set $G{=}[0, \Lambda] \times [\lambda-\tau, \lambda-\tau +
\Lambda]$)
\[ \H(\lambda - \tau, \Lambda) = \int_{\R^d} e^{\id^{\top} \vk{x}}
\pk{\exists (t,s) \in G:\ \vk{Y}(t) - \d_{\w}(t) > \vk{x},
	\vk{Y}(s) - \d_{\w}(s) > \vk{x}} 	d \vk{x},
\]
with $\d_{\w}(t) =\SA{t}{V_{\w}} \id =  |t|^{\alpha} \diag{\w} V \w.$
By Corollary~\ref{C:statn0} we have the following upper bound
$$\H(\lambda-\tau,\Lambda) \leq c_2 \Lambda^2 \exp \left( -c_3 (\lambda -
\tau - \Lambda)^{\alpha} \right), $$
which together with \eqref{e:lambdalarge} establishes  the proof.  \QED

\medskip

\begin{proof}[Proof of Theorem~\ref{T:nonstationary}]
	For  $\delta_u = u^{-2/\beta} \log^{2/\beta} u$ and  some $\theta>0$ sufficiently small by Assumption (D\ref{I:D2})
	\[ \max_{t \in [\delta_u, \theta]} \sigma_{\b}(t) \leq \sigma_{\b}^2(t_0) + c
	\delta^{\beta}_u = \sigma^2_{\b}(t_0) + c u^{-2} \log^2 u, \]
	which together with the vector-valued version of  Piterbarg inequality derived in Lemma~\ref{GBorell}  yields
	\begin{equation}
	\label{e:outside}
	\pk{\exists t \in [\delta_u, \theta]:\ \X(t) > u \b} \leq C T u^{2/\gamma
		- 1} \exp \left( - \frac{u^2}{2} \left(\sigma^2_{\b}(t_0) + c_2 u^{-2} \log^2 u \right)
	\right)
	\end{equation}
	and thus it suffices to consider the asymptotics of $\pk{\exists [0, \delta_u]:\ \X(t) > u \b}$ as $u\to \IF$.
	
Recall that in our notation
	$$\w = \Sigma^{-1} \tilb,  \quad V = A D A^{\top},  \quad W=  \Xi A^\top, \quad V_{\w} = \diag{\w} V \diag{\w}$$
and
$H_{\Y,\vk d_{\w}  }([0, \Lambda])= {\mathcal{H}}_{\alpha, V_{\w}} ([0,\Lambda])$  with $\d_{\vk w}= \SA{t}{V_{\w}} \id$ is the constant defined in \eqref{e:pickandsE}.

	{We divide the rest of the proof into  three separate cases.} 

  \medskip
	\textbf{1. Case $\beta > \alpha$.}
	For given $\Lambda$, $S$ positive  we define $N_u = \floor{S u^{\frac{2}{\alpha}-\frac{2}{\beta}}/\Lambda}$.
Consider a Gaussian process $\X_{u,k}(t) = \X(u^{-2/\alpha}(k \Lambda + t))$ with cmf $R_{u,k}(t,s) = R(u^{-2/\alpha}(k \Lambda + t), u^{-2/\alpha} (k \Lambda + s))$.  	By Assumption (D\ref{I:D2})   (set $\Sigma_{u,k} = R_{u,k}(0,0)$)
	\[ u^2 \left( \Sigma - \Sigma_{u,k} \right) \sim  \left( u^{\frac{2}{\beta} -
		\frac{2}{\alpha}} k \Lambda \right)^{\beta} \left( A \Xi^{\top} + \Xi
	A^{\top} \right) , \quad u\to \IF. 
	\]
	Hence,  $\tau_{\vk w}= \vk{w}^{\top}  A \Xi^{\top} \vk w > 0$ implies further 
	\begin{equation}
	\label{e:XukX0}
	\pk{\X_{u,k}(0) > u \b} = \pk{\X(t_0) > u \b }
	\exp \left( -\tau_{\vk w}
	\left( u^{\frac{2}{\beta} -
		\frac{2}{\alpha}} k \Lambda \right)^{\beta}  (1+o(1)) \right).
	\end{equation}
	
In view of both Assumptions (D\ref{I:D2})
	and (D\ref{I:D3})
  \begin{equation}
    \label{e:varsubtau}
    \lim_{u \to \infty} u^2 \left[ \Sigma_{u,k} - R_{u,k}(t,0) \right]
    =
    |t|^{\alpha} V
  \end{equation}
	and
	\begin{equation}
	\label{e:covsub}
	\lim_{u \to \infty} u^2 \left[ R_{u,k}(t,s) - R_{u,k}(t,0)
	\Sigma_{u,k}^{-1} R_{u,k}(0,s) \right]
	= R_{\alpha, V}(t,s)
	\end{equation}
	uniformly for all non-negative integers $ k \leq N_u$ and $t$,
	$s \in [0,
	\Lambda]$. 
	Define for $k\inn, u>0$
	$$A_k = \{\exists t \in [k \Lambda u^{-2/\alpha}, (k+1) \Lambda u^{-2/\alpha}]:\ \vk{X}(t) >
	u \vk{b} \} .$$
	In view of \eqref{e:varsubtau} and \eqref{e:covsub}, we have from  Lemma~\ref{Lem1} that
	\begin{equation}
	\label{e:asymsubpic}
	\pk{A_k}
	\sim \H_{\Y,\d_{\w}}([0, \Lambda])\pk{\vk X_{u,k}(0)> u \vk b}
	,  \quad 
	u \to \infty
	\end{equation}
	uniformly for all non-negative integers  $k \leq N_u$. 
 	This together with \eqref{e:XukX0} implies that
	\begin{align*}
    \frac{\sum_{k=0}^{N_u} \pk{A_k}}{\pk{\X(t_0) > u \b}}
	\sim&
	\frac{{\mathcal{H}}_{\alpha, V_{\w}} ([0,\Lambda])}{\Lambda} \; u^{\frac{2}{\alpha}-\frac{2}{\beta}}
	\,
	\sum_{k=0}^{N_u} \exp \left( -\tau_{\vk w}\left( u^{\frac{2}{\beta} -
		\frac{2}{\alpha}} k \Lambda \right)^{\beta}  (1+o(1))  \right) u^{\frac{2}{\beta} - \frac{2}{\alpha}} \Lambda \\
	\sim& \frac{{\mathcal{H}}_{\alpha, V_{\w}}([0,\Lambda])}{(\tau_{\vk w})^\beta \Lambda} u^{\frac{2}{\alpha}-\frac{2}{\beta}}
	\int_0^{S} e^{-x^{\beta}} d x 
 \end{align*}
as $u\to \IF$.
In view of  Lemma~\ref{L:largek}  for some $c_1 > 0$  we have
	\[
	\pk{\exists t \in [S u^{-2/\beta}, \delta_u]:\ \X(t) > u \b} \leq c_1 \exp \left( - \tau_{\w} \frac{S^{\beta}}{2} \right) \pk{\X(t_0) > u \b}.
	\]
	Letting $S \to \infty$ and then  $\Lambda \to \infty$ yields further
		\begin{equation}
	\label{e:subub}
	\limsup_{u \to \infty} \frac{\pk{\exists t \in [0, T]:\ \X(t) > u \b}}{u^{2/\alpha - 2 / \beta} \pk{\X(t_0) > u \b}} \leq  \HA_{\alpha, V_{\w}} \Gamma(1/\beta + 1) \tau_{\w}^{-\beta},
	\end{equation}
	where
	$$ {\mathcal{H}}_{\alpha, V_{\w}} = \lim_{\Lambda \to \infty} \frac{{\mathcal{H}}_{\alpha, V_{\w}} ([0,\Lambda])}{\Lambda} \in (0,\IF)$$
	 is defined in \eqref{e:Pickands}. Next,  we show the negligibility of  the double-sum term.
	Note first that
	\[ \lim_{u\to\infty} u^2 \left[ \Sigma_{u,k} - \Sigma_{u,k\Lambda,j\Lambda}
	\right] = \frac{1}{2} \left( j - k \right)^{\beta} \Lambda^{\beta}
	\]
	uniformly for all non-negative integers  $k < j \leq N_u$, where $\Sigma_{u,\tau,\lambda}$ is defined in \eqref{e:Sigmautl}.
	By Lemma~\ref{L:doublesub}, for some $c_2>0$
	\begin{equation}
	\label{e:doublekj}
	\begin{aligned}
	\sum_{j = k + 1}^{N_u} \pk{A_k A_j}
	\leq&
	c_2  \Lambda^2 u^{-|I|} \sum_{j=k+1}^{N_u} \exp \left( - \theta \left( j-k-1 \right)^{\alpha} \Lambda^{\alpha} \right) \varphi_{\Sigma_{u,k\Lambda,j\Lambda}}(u \tilb) \\
	\leq & c_2 \Lambda^2 \exp \left( -\theta \Lambda^{\beta} \right) u^{-|I|} \varphi_{\Sigma_{u,k}}(u \tilb),
	\end{aligned}
	\end{equation}
	which implies that the double-sum  $\sum_{0 \leq k < j \leq N_u} \pk{A_k A_j}$ is negligible compared to the single-sum if we let  $\Lambda \to \infty$. Therefore we complete the proof of \eqref{e:asysub}.
	
\medskip	 	
	\textbf{2. Case $\beta = \alpha$. }
{Let in the following $V_{\w}$, $W_{\w}$, $\vk f_{\w}$ be as in \eqref{fw}.
}
	It is {straightforward} to see from Lemma~\ref{Lem1} that
	\begin{equation}
	\label{e:nonc}
	\pk{\exists t \in [0, \Lambda u ^{-2/\alpha}]:\ \X(t) > u \b}
	\sim
  \H_{\Y, \d_{\w}+ \vk{f}_{\w}} ([0,\Lambda]) \pk{\X(t_0) > u \b}.
	\end{equation}
Applying Lemma~\ref{L:largek} we obtain
	\begin{equation}
	\label{e:nonAk}
	\pk{\exists t \in [\Lambda u^{-2/\alpha}, \delta_u]:\ \X(t) > u \b} \leq c_3 e^{-c_4 \Lambda^{\alpha}}.
	\end{equation}
	Combining \eqref{e:outside}, \eqref{e:nonc} and \eqref{e:nonAk} and letting $\Lambda \to \infty$ the claim in  \eqref{e:critical} follows {utilising  \nekorr{C:HYdL}}. \\
	\medskip
	\textbf{3. Case $\beta < \alpha$.} The proof is similar to the case  $\beta = \alpha$.
	Define $\X_u(t)	= \vk{X}( u^{-2/\beta} t)$, $u>0$  with cmf $R_u(t, s) = R( u^{-2/\beta} t,  u^{-2/\beta} s)$.
	By Assumptions (D\ref{I:D2}) and (D\ref{I:D3})
  \[
    \lim_{u \to \infty} u^2 \left[ \Sigma - R_u(t,0) \right]
    =
    |t|^{\beta} \Xi A^{\top}
  \]
	and
$$	\lim_{u \to \infty} u^2 \left[ R_u(t,s) - R_u(t,0) \Sigma^{-1} R_u(0,s)
	\right] = \vk{0}.
$$
	By Lemma~\ref{Lem1}, as $u\to \IF$
	\[
	\pk{\exists t \in [0, \Lambda u^{-2/\beta}]:\ \X(t) > u \b} \sim \H_{\vk{0},\vk f_{\w}}([0, \Lambda]) \pk{\X(t_0) > u \b},
	\]
	where
	\[
	\H_{\vk{0},\vk{f}_{\w}}([0, \Lambda])
	=
	\int_{\R^d} e^{\id^{\top} \x} \id_{\{\exists t \in [0, \Lambda]:\ - {\vk{f}_{\w}} |t|^{\beta} > \x \}} d \x.
	\]
	Further  Lemma~\ref{L:integral} implies
	\[
	\lim_{\Lambda \to \infty} \H_{\vk{0}, \vk f_{\w}}([0, \Lambda])
	=
	1 + \frac{\sum_{i \in I} w_i \max(0, - ( \Xi A^{\top} \w )_i)}{\w^{\top} \Xi A^{\top} \w},
	\]
	which together with Lemma~\ref{L:largek} yields \eqref{e:asymsup}.
\end{proof}

\section{Appendix}
We  present the proofs of 
{\eqref{e:Vbmin}}, Lemmas~\ref{AL}, \ref{lemGausUnif}, \ref{GBorell} and \ref{Lem1}. 

{
  \begin{proof}[Proof of \eqref{e:Vbmin}]
    In view of Assumption~(D\ref{I:D2}) we have as $t\to t_0$
    \[
      \Sigma^{-1}(t) - \Sigma^{-1} = \Sigma^{-1} \left( \Sigma - \Sigma(t) \right) \Sigma^{-1}(t)
      = \left| t - t_0 \right|^{\beta}  \Sigma^{-1} \left( A \Xi^{\top} + \Xi A^{\top} \right) \Sigma^{-1} + o \left( \left| t - t_0 \right|^{\beta} \right),
    \]
    where $A = A(t_0)$. Let $\widetilde{\b}(t)$ be the unique solution to the quadratic programming problem~\eqref{e:QP} with $\Sigma$ replaced by $\Sigma(t)$. Then we have from the fact that $\sigma_{\b}^{-2}(t) = \min_{\x \geq \b} \x^\top \Sigma^{-1}(t) \x$,
    \[
      \widetilde{\b}(t)^\top \left( \Sigma^{-1}(t) - \Sigma^{-1} \right) \widetilde{b}(t)
      \leq
      \sigma_{\b}^{-2}(t) - \sigma_{\b}^{-2}(t_0)
      \leq
      \widetilde{b}^\top \left( \Sigma^{-1}(t) - \Sigma^{-1} \right) \widetilde{b}.
    \]
    This and the fact that $\widetilde{b}(t)$ is Lipschitz continuous (c.f. Lemma~\ref{L:lip}) complete the proof.
  \end{proof}
}

\medskip
\prooflem{AL}
All the claims apart from \eqref{eq:alfaB} are known, see e.g., \cite[Lem 2.1]{Rolski17}.
Repeating the arguments of \cite[Lem 1]{Debicki10} (therein $\vk b> \vk 0$ is assumed) we have for any $\vk z \ge \vk 0$ with $\vk z^\top \b\ge 0$ and $B$ a square matrix such that $BB^\top = \Sigma$
$$ 0\le \vk z^\top \vk b = \inf_{ \vk x\ge \vk b} \vk z^\top \vk x \le
\abs{B^\top \vk w} \inf_{ \vk x \ge \vk b} \abs{B^{-1} \vk x} =
\sqrt{\vk{z}^\top \SI \vk{z}} \inf_{ \vk x \ge \vk b} \vk x^\top \Sigma^{-1} \vk x  $$
implying thus
$$ \max_{ \vk{z}\in [0,\IF)^d: \vk{z}^\top \b> 0} \frac{(\vk{z}^\top \b)^2 }{\vk{z}^\top \SI \vk{z}} \le
\min_{\x \ge  \b}
\x^\top \SIM\x.$$
By the properties of the unique solution $\tilb$ of $\Pi_{\Sigma}(\vk b)$ we have  that the unique  solution $\vk w$ of the dual programming problem of $\Pi_\Sigma(\vk b)$ is given by  $\vk{w}= \SIM \tilb$. We have  $\tilb_I=\vk b_I$ and  if $J$ is non-empty,
then   $ \tilb_J= \Sigma_{JI} \Sigma_{II}^{-1 } \vk b_I$. {Since $\Sigma$ is non-singular, then
$ (\Sigma^{-1})_{JJ} \Sigma_{JI} = -(\Sigma^{-1})_{JI}\Sigma_{II}$. Consequently, we obtain
$$ \vk w_I = (\Sigma^{-1})_{II} \vk b_I+ (\Sigma^{-1})_{IJ} \Sigma_{JI}( \Sigma_{II})^{-1 } \vk b_I
= [ (\Sigma^{-1})_{II} + (\Sigma^{-1})_{IJ} \Sigma_{JI} (\Sigma_{II})^{-1 } ]\vk b_I= (\Sigma_{II})^{-1} \vk b_I
$$
and
$$ \vk w_J = (\Sigma^{-1})_{JI} \vk b_I+ (\Sigma^{-1})_{JJ} \Sigma_{JI} (\Sigma_{II})^{-1 } \vk b_I =
[ (\Sigma^{-1})_{JI} + (\Sigma^{-1})_{JJ} \Sigma_{JI} ()\Sigma_{II})^{-1 } ]\vk b_I =\vk 0_J
$$
}
hold implying {that} 
$$ \vk w^\top \vk b =
\vk w_I ^\top \vk b_I= \b_I^\top \SIIIM \b_I =\inf_{ \vk x \ge \vk b} \vk x^\top \Sigma^{-1} \vk x>0 ,$$ which yields  further
$$ \frac{(\vk{w}^\top \b)^2 }{\vk{w}^\top\SI \vk{w}}= \vk b_I^\top \SIIIM \vk b_I= \min_{\x \ge  \b}
\x^\top \SIM\x.$$
Hence \eqref{eq:alfaB} follows and  the proof is complete.
\QED

\medskip
\prooflem{lemGausUnif}
First note that a family of vector-valued processes is tight if its components are tight, see e.g.,  \cite[Cor 1.3]{HaL}.
Let $\tau_u\in Q_u,u>0$ be given and set $\vk Z_u(t)= \vk X_{u,\tau_u}(t)- \vk f_{u,\tau_u}(t)$.
For given  $s_i\inr, ,t_i \in E, i\le n$ by
Berman's comparison lemma (see e.g., \cite{Berman92}) and \eqref{covcon} for some $c>0$ 
\BQNY
\lefteqn{ \hspace{- 2 cm}
	\abs{ \pk{ X_{i,u,\tau_u}(t_k) - f_{i,u,\tau_u}(t_k) \le s_i,i \le d, k\le n}- \pk{ Y_{i}(t_k) - f_{i,u,\tau_u}(t_k) \le s_i,i \le d, k\le n }}  }\\
&\le & c \sum_{1 \le l \le k \le n, 1 \le i\le j \le n} \abs{  \cov(X_{i,u,\tau_u}(t_l), X_{j,u,\tau_u}(t_k)  )- \cov(Y_i(t_l),Y_j(t_k))}  \\
&\to & 0, \quad u\to \IF. \quad
\EQNY
Since $Y_i(t_k), i\le d, k\le n$ has a continuous distribution, {then} \eqref{uniftauT} yields
\BQNY
\limit{u}	\abs{ \pk{ Y_{i}(t_k) - f_{i,u,\tau_u}(t_k) \le s_i,i \le d, k\le n } -
	\pk{ Y_{i}(t_k) - f_{i}(t_k) \le s_i,i \le d, k\le n }  }  &= & 0 .
\EQNY
Consequently, the  fidi's  of $\vk Z_u$ converge in distribution to those of $\vk Y - \vk f$ as $u\to \IF$.

Moreover, condition \eqref{uniftauT} implies that each component of $\xxiu$, $u>0$ is tight, see \cite[Prop 9.7]{Pit20}. By \eqref{uniftauT} each component of $\vk Z_u, u>0$ is also tight, and thus $\vk Z_u, u>0$ is tight.
Since by assumption $\Gamma$ is continuous, the continuous mapping theorem implies that for any continuity point
$s$ of $\Gamma(\vk Y- \vk f)$ we have
\BQN \label{eq:contra} \limit{u} \pk{ \Gamma( \vk Z_u) > s} = \pk{\Gamma(\vk Y - \vk f)> s},
\EQN
hence
\BQNY \limit{u} \sup_{ \tau \in Q_u} \abs{ \pk{ \Gamma(\xxiu- \vk{f}_{u,\tau})> s} - \pk{\Gamma(\vk{Y}- \vk{f})>s}}=0.
\EQNY
Indeed, if the above is not satisfied, then for a given $\ve>0$ and all $u$ large we can find $\tau_u$ such that
$\abs{ \pk{ \Gamma( \vk X_{u,\tau_u} - \vk{f}_{u,\tau_u})> s} - \pk{\Gamma(\vk{Y}- \vk{f})>s}}> \ve$
which is a contradiction in view of \eqref{eq:contra}, hence the proof is complete. \QED

\medskip
{
\prooflem{l.cont}
Let
 $\vk x \in \R^d$ and
$\vk \delta =(\delta_1,...,\delta_d)^\top\ge \vk 0 $ be given. The proof follows from combination of the fact that
\begin{eqnarray*}
\lefteqn{
\pk{\exists
		t \in E: \ \vk Y (t) > \x -\delta}
-
\pk{\exists
		t \in E: \ \vk Y (t) > \x +\delta}}\\
&\le&
\sum_{i=1}^d
\pk{\sup_{t \in E}  Y_i (t) \in (x_i-\delta_i,x_i+\delta_i]}
\end{eqnarray*}
and Tsirelson Theorem (see, e.g., \cite[Thm 7.1]{AzW09} and \cite{MR666093})
which implies that
$\pk{\sup_{t \in E}  Y_i (t) \le x}$
is continuous  $i=1,\ldots,d$ except at most at one point $s_i=\inf\{s:\pk{\sup_{t \in E}  Y_i (t) \le x}>0\}$.
\QED
}

\medskip
\prooflem{GBorell}
For $\vk{z} \in [0,\IF)^d$ such that $\vk{z}^\top \b > 0$, we define
$Y_{\vk{z}}(t) =
\widetilde{\vk{z}}^\top \vk{Z}(t)$ with $\widetilde{\vk{z}} = \vk{z}/(\vk{z}^\top \b)$.
Clearly,  $\{ \vk{Z}(t) > u \b \} \subset \{ Y_{\vk{z}}(t) > u \}$ for such $\vk{z}$ implying
\BQN\label{eq:BorellY}
\pk{\exists {t\in\VV} :\ \vk{Z}(t)>u\vk{b}}&\le& \inf_{\vk{z} \in
	[0,\IF)^d : \vk{z}^\top \vk{b}>0} \pk{ \sup_{t\in\VV}
	{Y}_{\vk{z}}(t) > u }.
\EQN
The proof of \eqref{eqBO} follows by a direct application of Borell-TIS inequality to
$Y_{ \vk{z}}(t) $ (see e.g., \cite{AdlerTaylor}) with
$$\mu=\E{\sup_{t\in \VV} Y_{\vk z}(t)  }<\IF.$$

Next, if $\Sigma(t)$ is non-singular for $t\in \VV$, choose $\vk{v}(t)\ge \vk{b}$ such that it minimises $\vk{v}^\top(t) \Sigma^{-1}(t) \vk{v}(t)$ and
set $\vk{z}(t) = \Sigma^{-1}(t) \vk{v}(t)$. By \nelem{AL} and \eqref{eq:ZG2}
$$ \sup_{\vk{z}\ge [0,\IF)^d: \vk{z}^\top \vk{b}>0} \Var( Y_{\vk{z}}(t)) = \Var( Y_{\vk{z}(t)}(t))
=\frac{1}{ \vk{v}(t)^\top \Sigma^{-1}(t)\vk{v}(t)}>0.$$
In view of Lemma \ref{L:lip}, \eqref{e:holder} and the compactness of $\VV$ for $\vk f=\vk v$ or $\vk f= \vk w$ we obtain
\BQN \label{ef}
\abs{ \vk f(s)-  \vk f(t)} \le C_1 \abs{t- s}^{\gamma}
\EQN
for some positive constant $C_1$. 
Note that since $\VV$ is compact
\BQN
\inf_{t \in \VV} \vk{v}(t)^\top \b > 0.
\EQN
It follows that $\vk f(t)=\widetilde{\vk{z}}(t) = \vk{z}(t) / [\vk{z }(t)^\top \b]$ also satisfies \eqref{ef} (for some other constant $C_1$).
This and \eqref{eq:ZG2} imply the $\gamma$-H\"older continuity of
$Y_{\widetilde{\vk{z}}}(t) = \widetilde{\vk{z}}(t)^\top \vk{Z}(t)$.
Therefore \eqref{IPIT} follows by applying \cite[Thm ~8.1]{Pit96} to $Y_{\widetilde{\vk{z}}}$.


If $\VB(t),t\in \VV$ has a unique maximum at $t_0\in \VV$ and is continuous, the claim follows by using first Borell-TIS inequality
and then applying Piterbarg inequality for the neighborhood of $t_0$,
see the derivation of (32) and (33) in \cite{DHJT15}. \QED

\medskip
\prooflem{Lem1}
For simplicity we assume that the index set $J$ is not empty.
Define  $\uY \in \R^d$ such that $\uY_I$ has all components equal $u$ and $\uY_J$ has all components equal 1.
Set next
\BQNY
\vk{Z}_{u,\tau}(t)=\uY \cdot \bigl [\xxiu(\t) -u \tilb \bigr] + \x,
\quad \vk{\chi}_{u,\tau}(t)= \bigl(\vk{Z}_{u,\tau}(t) \lvert \vk{Z}_{u,\tau}(0)= \vk{0} \bigr),
\EQNY
{where} in our notation $\vk x \cdot \vk y=(x_1y_1 \ldot x_d y_d)^\top$ for $\vk x$ and $\vk y$ two vectors in $\R^d$.
For any $u>0$ we obtain
\begin{align*}
& \pk{\exists {t \in \bD}:\ \xxiu(t) > u \vk{b}} \\
=&
\pk{\exists {t \in \bD}:\ \bigl( \xxiu(\t) -u \tilb\bigr)> u (\b - \tilb)}\\
=&u^{-m} \int_{\R^d}
\pk{\exists t \in \bD :\ \bigl(  \xxiu(\t) -u \tilb \bigr)> u (\b - \tilb)
	\big|
	\xxiu(0)= u \tilb  - \vk{x}/\uY } \varphi_{\Sigma_{u,\tau}}( u \tilb -
\vk{x}/ \uY )\, d\x \\
=& u^{-m} \int_{\R^d}
\pk{\exists t \in \bD :\  \vk{Z}_{u,\tau}(t) > \x + \uY u( \b-\tilb)  \big|
	\vk{Z}_{u,\tau}(0)=\vk{0}} \varphi_{\Sigma_{u,\tau}}( u \tilb - \vk{x}/ \uY )\, d\x\\
=& u^{-m} \int_{\R^d}
\pk{\exists t \in \bD :\  \bigl( \vk{\chi}_{u,\tau}(t) -  u \uY ( \b-
	\tilb) \bigr)> \x } \varphi_{\Sigma_{u,\tau}}( u \tilb - \vk{x}/
\uY )\, d\x .
\end{align*}

{Assumption} (A\ref{I:A2}) implies
\begin{equation}
\label{e:echi}
\begin{aligned}
\E{\vk{\chi}_{u,\tau}(t) }
= & \uY \cdot \left\{ R_{u,\tau}(t, 0) \Sigma_{u,\tau}^{-1} (u \tilb - \x/\uY) - u \tilb \right\} + \x \\
= & \begin{pmatrix}
u^2 \left\{ (R_{u,\tau}(t,0) - \Sigma_{u,\tau} ) \Sigma_{u,\tau}^{-1} \tilb \right\}_I \\
u \left\{ (R_{u,\tau}(t,0) - \Sigma_{u,\tau} ) \Sigma_{u,\tau}^{-1} \tilb \right\}_J
\end{pmatrix}
+ \uY \cdot \left\{ (R_{u,\tau}(t,0) - \Sigma_{u,\tau} ) \Sigma_{u,\tau}^{-1} (\x / \uY) \right\} \\
\to & \begin{pmatrix}
- \vk{d}_I(t) \\
\vk{0}_J
\end{pmatrix}
\end{aligned}
\end{equation}
and
\begin{equation}
\label{e:vchi}
\begin{aligned}
& \E{ \left[\vk{\chi}_{u,\tau}(t) - \E{\vk{\chi}_{u,\tau}(t)} \right]
	\left[ \vk{\chi}_{u,\tau}(s) - \E{\vk{\chi}_{u,\tau}(s)} \right]^\top } \\
= & \diag{\uY} \left[ R_{u,\tau}(t,s) - R_{u,\tau}(t,0) \Sigma_{u,\tau}^{-1} R_{u,\tau}(0,s) \right] \diag{\uY} \\
\to & \begin{pmatrix}
K_{II}(t,s) & \vk{0} \\
\vk{0} & \vk{0}
\end{pmatrix}
\end{aligned}
\end{equation}
uniformly in $t$, $s\in \bD$ and $\tau \in Q_u$ as $u \to \infty$. It follows that the assumptions of Lemma
\ref{lemGausUnif} hold true with $\xxiu(t)$, $\vk{f}_{u,\tau}(t)$ and $\vk{f}(t)$ replaced by
$\vk{\chi}_{u,\tau}(t) - \E{\vk{\chi}_{u,\tau}(t)}$,
$\E{\vk{\chi}_{u,\tau}(t)}$
and $\vk{d}(t)$.
{Define the index set $L$ as the {maximal} subset of $J$ such that $\tilde b_i= b_i$ for any $i \in L$. Hence  $\tilde{b}_i= b_i$ for all $i\in I\cup L$ and $\tilde{b}_i>b_i$ for $i\in J \setminus L$. For simplicity we shall assume that $L$ is non-empty.}  If $\W(t), t\ge 0$ denotes a  centered Gaussian process with $K$ as its cmf, then $\vk{\chi}_{u,\tau}(t) -
\E{\vk{\chi}_{u,\tau}(t)}$, $t\in \bD$ converges  for any $\tau \in Q_u$ in distribution to $\begin{pmatrix}
\W_I(t) \\ \vk{0}_J \end{pmatrix}$ on
the Banach space $C(E)$ (the space of all $\R^d$-valued continuous functions on $E$) equipped with the sup-norm.  We write next
\[
	\pk{\exists t \in \bD :\  ( \vk{\chi}_{u,\tau}(t) -  u \uY (
    \b- \tilb) )> \x } =
  \pk{ \sup_{t\in  \bD} \min_{1 \le i \le d}  \Bigl(\vk{\chi}_{u,\tau,i}(t) - u \overline{u}_i (b_i - \tilde{b}_i) -  x_i \Bigr)> 0 }=:a_u(\x). \]
{Since
$\sup_{t\in \bD} \min_{1\le i \le d} f_i, \vk f\in C(E)$
is a continuous functional on $C(E)$, then by \nelem{lemGausUnif} and \nelem{l.cont}}
\begin{equation}
\label{e:convprobx}
\limit{u} \sup_{\tau \in Q_u}\Abs{
	a_u(\x)- \id_{\{\x_L < \vk{0}_L\}} \pk{\exists
		t \in \bD: \ \W_I(t) - \vk{d}_I(t) > \x_I} }=0
\end{equation}
holds for almost all $\vk x\inr^d$.  Note that, by Lemma~\ref{AL}  there exists a positive constant $\lambda$
such that for all $u$ large and all $\tau \in Q_u$
\[ \frac{1}{2} (\x /\uY)^\top \SI^{-1}_{u,\tau} (\x/\uY) \geq \lambda \left(
\vk{x}/\uY \right)^{\top} \left( \vk{x}/\uY \right) \geq \lambda
\vk{x}_J^{\top} \vk{x}_J.
\]
Hence  (set $\vk w= \Sigma^{-1} \tilb$ and recall that
$\vk w_I= \SIIIM \b_I> \vk 0_I$, $\vk w_J= \vk 0_J$ and $\vk w^\top \x= \vk w_I^\top \vk x_I$ for any $\x \inr^d$)
\begin{align*}
\varphi_{\Sigma_{u,\tau}}( u \tilb  - \vk{x}/\uY)
= & \varphi_{\Sigma_{u,\tau}}( u \tilb) \exp \left( u \left( \x/\uY
\right)^{\top} \Sigma_{u,\tau}^{-1} \tilb - \frac{1}{2} \left( \x/\uY
\right)^{\top} \Sigma_{u,\tau}^{-1} \left( \x/\uY \right) \right) \\
= & \varphi_{\Sigma_{u,\tau}}( u \tilb)  \exp \left( \vk{w}_I^{\top} 	\vk{x}_I - \frac{1}{2} (\x /\uY)^\top \SI^{-1}_{u,\tau} (\x/\uY) ) + u
\tilb^{\top} \left( \Sigma_{u,\tau}^{-1} - \Sigma^{-1} \right)
\left( \x/\uY \right)^{\top} \right) \\
\leq & \varphi_{\Sigma_{u,\tau}}(u\tilb) \exp \left(\vk{w}_I^{\top} \vk{x}_I + u \tilb^{\top} \left(
\Sigma_{u,\tau}^{-1} - \Sigma^{-1} \right) \left( \vk{x}/\uY
\right) - \lambda \vk{x}_J^{\top} \vk{x}_J \right).
\end{align*}
In view of  \eqref{e:convprobx} and Assumption (A\ref{I:A1}) as $u\to \IF$
\begin{equation}
\label{e:asypick}
\begin{aligned}
&  \pk{\exists t \in \bD :\  \xxiu(t) > u \vk{b}} \\
\sim & u^{-m} \varphi_{\Sigma_{u,\tau}}(u \tilb)
\int_{\R^d}
e^{ \x^{\top}_I  \vk{w}_I- \frac{1}{2} \x_J^\top
	(\SI^{-1})_{JJ}\x_J}
\id_{\{\x_L < \vk{0}_L\}}
\pk{\exists {t \in \bD}:\ \W_I(t) - \vk{d}_I(t)  > \x_I} \, d\x.
\end{aligned}
\end{equation}
The asymptotic equivalence  above follows by the dominated convergence theorem. We
shall justify its {applicability} as follows.
First note that 
$$ w_i=(\SIIIM \vk{b}_I)_i >0, \quad \forall i\in I.$$
Split the region of the integration on sets where the $i$th
coordinates of $\x_I$ are either positive or negative.
If $x_i<0$, then the domination {of the integrand} for this coordinate is clear since $w_i  >0$. 
Suppose that we deal therefore with a region $\Omega$ where the first $l$ components are negative and the $m-l$ components in the index set $F$ are positive.
Then (recall the definition of $a_u(\x)$ above)
\begin{eqnarray*}
a_u(\x)
\leq  \pk{\exists t \in \bD:\ \vk{\chi}_{u,\tau,F}(t) > \vk{x}_F}
\leq  \pk{\exists t \in \bD: \ \vk{1}_F^{\top}
	\vk{\chi}_{u,\tau,F}(t) > \vk{1}_F^{\top} \vk{x}_F}.
\end{eqnarray*}
It follows from \eqref{e:echi}, \eqref{e:vchi} and Assumptions
(A\ref{I:A1}), (A\ref{I:A2}) that for any $t\in E$
\[ \E{\vk{1}_F^{\top} \vk{\chi}_{u,\tau,F}(t)} \leq c + \delta
\vk{1}_F^{\top} \vk{x}_F \quad \text{and }
\E{( \vk{1}_F^{\top} \vk{\chi}_{u,\tau,F}(t) )^2} \leq
\sigma^2 \]
for some $\delta > 0$, $\sigma > 0$ and some constant $c$.
By the Borell-TIS inequality (c.f. \cite[Theorem~D.1]{Pit96})
\BQNY
a_u(\x)&\leq & \pk{\exists t \in \bD: \ \vk{1}_F^{\top}
	\vk{\chi}_{u,\tau,F}(t) > \vk{1}_F^{\top} \vk{x}_F} \leq  e^{-\varepsilon \left( \vk{x}_F^{\top} \vk{x}_F \right)}
\EQNY
for some $\ve>0$ small enough, hence the domination of the integrand follows.

Taking in particular $E = \{0\}$ in \eqref{e:asypick} implies as $u\to \IF $
\[
\pk{\X_{u,\tau}(0) > u \b} \sim u^{-m} \varphi_{\Sigma_{u,\tau}}(u \tilb) \left( \prod_{i \in I} w_i \right)^{-1} \int_{\R^{|J|}} e^{- \frac{1}{2} \x_J^{\top} \left( \Sigma^{-1} \right)_{JJ} \x_J} \id_{\{\x_L < \vk{0}_L\}} d \x_J.
\]
Therefore, as $u\to \IF$
\begin{align*}
& \pk{\exists t \in E:\ \X_{u,\tau}(t) > u \b} \\
\sim &
\pk{\X_{u,\tau}(0) > u \b} \left( \prod_{i \in I} w_i \right) \int_{\R^m} e^{\w_I^{\top} \x_I} \pk{ \exists t \in E:\ \W_I(t) - \d_I(t) > \x_I} d \x_I,
\end{align*}
which complete the proof.
\QED

\section*{Acknowledgments}
Partial 	support by SNSF Grant 200021-175752/1  is kindly acknowledged.
K.D. was partially supported by NCN Grant No
2018/31/B/ST1/00370.

\bibliographystyle{ieeetr}
\def\polhk#1{\setbox0=\hbox{#1}{\ooalign{\hidewidth
  \lower1.5ex\hbox{`}\hidewidth\crcr\unhbox0}}}

\end{document}